\documentclass[ejs,preprint,dvips,twoside]{imsart}
\RequirePackage[colorlinks,citecolor=blue,urlcolor=blue]{hyperref}
\RequirePackage[OT1]{fontenc}
\RequirePackage{amsthm,amsmath,amsfonts,amssymb}
\RequirePackage[numbers]{natbib}

\usepackage{graphicx,subfigure}
\usepackage{bm} 
\usepackage{calligra} 
\usepackage{mathbh}
\usepackage{rotating}

\startlocaldefs
\numberwithin{equation}{section}

\theoremstyle{plain}
\newtheorem{theorem}{Theorem}[section]
\newtheorem{lemma}[theorem]{Lemma}
\newtheorem{proposition}[theorem]{Proposition}

\theoremstyle{definition}
\newtheorem{example}[theorem]{Example}

\newcommand{\nn}{\nonumber}
\newcommand{\iy}{\infty}
\newcommand{\f}[2]{\frac{#1}{#2}}

\newcommand{\mc}[1]{\mathcal{#1}}

\newcommand{\half}{\f{1}{2}}
\newcommand{\rec}[1]{\frac{1}{#1}}
\newcommand{\Ind}{\mathbh{1}}
\newcommand{\sumin}{\sum_{i=1}^{n}}
\newcommand{\eqnref}[1]{$(\ref{#1})$}
\newcommand{\defMatrix}[1]{(\!(#1)\!)}
\newcommand{\eig}{\mathrm{eig}}

\renewcommand{\P}{\mathrm{P}}
\newcommand{\E}{\mathrm{E}}
\newcommand{\tr}{\mathrm{tr}}
\newcommand{\diag}{\mathrm{diag}}

\newcommand{\al}{\alpha}
\newcommand{\ga}{\gamma}
\newcommand{\de}{\delta}
\newcommand{\ve}{\varepsilon}
\newcommand{\ep}{\epsilon}

\newcommand{\Om}{\Omega}

\newcommand{\bX}{\bm{X}}
\newcommand{\bZ}{\bm{Z}}
\newcommand{\balpha}{\bm{\alpha}}
\newcommand{\bbeta}{\bm{\beta}}
\newcommand{\bGamma}{\bm{\Gamma}}
\newcommand{\bSigma}{\bm{\Sigma}}
\newcommand{\bOmega}{\bm{\Omega}}
\newcommand{\bmS}{\bm{S}}
\newcommand{\bmA}{\bm{A}}
\newcommand{\bmB}{\bm{B}}
\newcommand{\bmD}{\bm{D}}
\newcommand{\bmI}{\bm{I}}

\newcommand{\RR}{{\mathbb R}}

\mathchardef\given="626A
\mathcode`:="603A
\long\def\beginskip#1\endskip{}

\endlocaldefs

\doi{10.1214/14-EJS945}
\pubyear{2014}
\volume{8}
\issue{2}
\firstpage{2111}
\lastpage{2137}

\allowdisplaybreaks

\begin{document}

\begin{frontmatter}
\title{Posterior convergence rates for estimating large precision
matrices using graphical models\thanksref{T1}}
\runtitle{Convergence rates for estimating large precision matrices}
\thankstext{T1}{Research is partially supported by NSF grant number
DMS-1106570.}

\begin{aug}
\author{\fnms{Sayantan} \snm{Banerjee}\corref{}\ead[label=e1]{SBanerjee@mdanderson.org}}
\address{Department of Biostatistics\\
The University of Texas MD Anderson Cancer Center\\
1400 Pressler Street\\
Houston, TX 77030\\
USA\\
\printead{e1}}
\end{aug}
\smallskip\textbf{\and}\vspace*{-6pt}
\begin{aug}
\author{\fnms{Subhashis} \snm{Ghosal}\ead[label=e2]{sghosal@stat.ncsu.edu}}
\address{Department of Statistics\\
North Carolina State University\\
4276 SAS Hall, 2311 Stinson Drive\\
Raleigh, NC 27695-8203\\
USA\\
\printead{e2}}

\runauthor{S. Banerjee and S. Ghosal}
\end{aug}

\begin{abstract}
We consider Bayesian estimation of a $p\times p$ precision matrix, when
$p$ can be much larger than the available sample size $n$. It is well
known that consistent estimation in such ultra-high dimensional
situations requires regularization such as banding, tapering or
thresholding. We consider a banding structure in the model and induce a
prior distribution on a banded precision matrix through a Gaussian
graphical model, where an edge is present only when two vertices are
within a given distance. For a proper choice of the order of graph, we
obtain the convergence rate of the posterior distribution and Bayes
estimators based on the graphical model in the $L_{\infty}$-operator
norm uniformly over a class of precision matrices, even if the true
precision matrix may not have a banded structure. Along the way to the
proof, we also compute the convergence rate of the maximum likelihood
estimator (MLE) under the same set of condition, which is of
independent interest. The graphical model based MLE and Bayes
estimators are automatically positive definite, which is a desirable
property not possessed by some other estimators in the literature. We
also conduct a simulation study to compare finite sample performance of
the Bayes estimators and the MLE based on the graphical model with that
obtained by using a Cholesky decomposition of the precision matrix.
Finally, we discuss a practical method of choosing the order of the
graphical model using the marginal likelihood function.
\end{abstract}



\received{\smonth{11} \syear{2013}}

\end{frontmatter}

\section{Introduction}
\label{sec:Intro}
Estimating a covariance matrix or a precision matrix (inverse
covariance matrix) is one of the most important problems in
multivariate analysis. Of special interest are situations where the
number of underlying variables $p$ is much larger than the sample size
$n$. These situations are common in gene expression data, fMRI data and
in several other modern applications. Special care needs to be taken
for tackling such high-dimensional scenarios. Conventional estimators
like the sample covariance matrix or maximum likelihood estimator
behave poorly when the dimension is much higher than the sample size.

Different regularization based methods have been proposed and
developed in the recent years for dealing with high-dimensional data.
These include banding, thresholding, tapering and penalization based
methods to name a few; see, for example, \cite{ledoit2004well,
huang2006covariance, yuan2007model, bickel2008covariance,
bickel2008regularized, karoui2008operator, friedman2008sparse,
rothman2008sparse, lam2009sparsistency, rothman2009generalized,
cai2010optimal, cai2011constrained}. Most of these regularization based
methods for high dimensional models impose a sparse structure in the
covariance or the precision matrix, as in \cite{bickel2008covariance},
where a rate of convergence has been derived for the estimator obtained
by ``banding'' the sample covariance matrix, or by banding the Cholesky
factor of the inverse sample covariance matrix, as long as $n^{-1}\log
p \rightarrow0$. Cai et al. \cite{cai2010optimal} obtained the minimax
rate under the operator norm and constructed a tapering estimator which
attains the minimax rate over a smoothness class of covariance
matrices. Cai and Liu \cite{cai2011adaptive} proposed an adaptive
thresholding procedure. More recently, Cai and Yuan \cite
{cai2012adaptive} introduced a data-driven block-thresholding estimator
which is shown to be optimally rate adaptive over some smoothness class
of covariance matrices.\looseness=1

There are only a few convergence results available in the Bayesian
setting for estimating large covariance or precision matrices. Ghosal
\cite{ghosal2000asymptotic} studied asymptotic normality of posterior
distributions for exponential families (which include the multivariate
normal scale family) when the dimension $p \rightarrow\infty$, but
restricting to the situation $p \ll n$. Recently, Pati et al. \cite
{pati2012factor} considered sparse Bayesian factor models for
dimensionality reduction in high dimensional problems and showed
consistency in the $L_2$-operator norm (also known as the spectral
norm) by using a point mass mixture prior on the factor loadings,
assuming such a factor model representation for the true covariance matrix.

Graphical models \cite{lauritzen1996graphical} provide an
excellent tool for sparse covariance or inverse covariance estimation;
see \cite{dobra2004sparse, meinshausen2006high, yuan2007model,
friedman2008sparse}, as they capture the conditional dependence between
the variables by means of a graph. Bayesian methods for inference using
graphical models have also been developed, as in \cite
{roverato2000cholesky, atay2005monte, letac2007wishart}. For a complete
graph corresponding to the saturated model, clearly the Wishart
distribution is the conjugate prior for the precision matrix $\bOmega$.
For an incomplete decomposable graph, a conjugate family of priors is
given by the $G$-Wishart prior \cite{roverato2000cholesky}. The
equivalent prior on the covariance matrix is termed as the hyper
inverse Wishart distribution in \cite{dawid1993hyper}. Letac and Massam
\cite{letac2007wishart} introduced a more general family of conjugate
priors for the precision matrix, known as the $W_{P_G}$-Wishart family
of distributions, which also has the conjugacy property. The properties
of this family of distribution were further explored in \cite
{rajaratnam2008flexible}. Rajaratnam et al. \cite
{rajaratnam2008flexible} also obtained expressions for Bayes estimators
under different loss functions.

In this paper, we consider Bayesian estimation of the precision
matrix working with a $G$-Wishart prior induced by a Gaussian graphical
model, which has a Markov property with respect to a decomposable graph
$G$. More specifically, we work with a Gaussian graphical model
structure which induces banding in the corresponding precision matrix.
Approximate banding structure for precision matrix can arise in certain
possibly non-stationary time series framework. Suppose that $\{X_t: t
=1,\ldots,p\}$ is a possibly non-stationary time series with
approximately Markov dependence on neighborhoods in the sense that
off-diagonal elements of its precision matrix decay sufficiently fast
with the lag. The covariances cannot be estimated based on a single
time series due to lack of stationarity. However, if we have
replications $\bm{X}_1,\ldots,\bm{X}_n$, even when $n$ is much smaller
than $p$, it is still possible to estimate the entire precision matrix
assuming the approximate Markov structure. The graphical model based on
the banding structure ensures the decomposability of the graph, along
with the presence of a perfect set of cliques, as explained in Section
\ref{sec:Notations}. For a $G$-Wishart prior, we can compute the
explicit expression of the normalizing constant of the corresponding
marginal distribution of the graph (see Section \ref{sec:bandestim}).
For arbitrary decomposable graphs, the computation of the normalizing
constant requires Markov chain Monte-Carlo (MCMC) based methods; see
\cite{atay2005monte, carvalho2007simulation, carvalho2009objective,
lenkoski2011computational, dobra2011bayesian}. We obtain posterior
convergence rate and convergence rate of the Bayes estimators and the
MLE for the graphical model based on banding on the precision matrix.
However, we allow the true precision matrix to be outside this class,
provided it is well-approximated by banded matrices in an appropriate sense.

The paper is organized as follows. In the next section, we discuss
some preliminaries on graphical models. In Section \ref
{sec:modelandprior}, we formulate the estimation problem and the
describe the corresponding model assumptions. Section \ref
{sec:mainresults} deals with the main results related to posterior
convergence rates. A method for selecting the banding parameter using
the explicit form of the marginal likelihood of a graph is discussed in
Section \ref{sec:bandestim}. In Section \ref{sec:sim}, we compare the
performance of the Bayesian estimators with that of the graphical
maximum likelihood estimator (MLE) and the banding estimator proposed
in \cite{bickel2008regularized}. Proofs of the results are presented in
Section \ref{sec:proofs}. Some auxiliary lemmas and their proofs are
given in the \hyperref[app]{Appendix}.

\section{Notations and preliminaries on graphical models}
\label{sec:Notations}
We first describe the notations to be used in this paper. By $t_n =
O(\delta_n)$ (respec\-tively, $o(\delta_n)$), we mean that $t_n/\delta_n$
is bounded (respectively, $t_n/\delta_n \rightarrow0$ as \mbox{$n\to\iy$}).
For a random sequence $X_n$, $X_n = O_P(\delta_n)$ (respectively, $X_n
= o_P(\delta_n)$) means that $\P(|X_n| \leq M\delta_n) \rightarrow1$
for some constant $M$ (respectively, $\P(|X_n| < \epsilon\delta_n)
\rightarrow1$ for all $\epsilon> 0$). For numerical sequences $r_n$
and $s_n$, by $ r_n \ll s_n$ (or,~$r_n \gg s_n)$ we mean that $r_n =
o(s_n)$, while by $s_n \gtrsim r_n$ we mean that $r_n = O(s_n)$. By
$r_n \asymp s_n$, we mean that $r_n = O(s_n)$ and $s_n = O(r_n)$, while
$r_n\sim s_n$ stands for $r_n/s_n\to1$. The indicator function is
denoted by $\Ind$.

We denote vectors by bold lowercase English or Greek letters. The
components of a vector are represented by the corresponding non-bold
letters, that is, for $\bm{x}\in\RR^p$, $\bm{x}= (x_1,\ldots
,x_p)^T$. We
define the following norms for a vector $\bm{x}\in\RR^p$:
$\|\bm{x}\|_r = (\sum_{j=1}^p|x_{j}|^r)^{1/r}$, $\|\bm
{x}\|_\infty
= \mathop{\max}_{j}|x_{j}|$. 
Matrices are denoted by bold uppercase English or Greek letters, like
$\bmA= (\!(a_{ij})\!)$, where $a_{ij}$ stands for the $(i,j)$th entry
of $\bm{A}$.
If $\bm{A}$ is a symmetric $p\times p$ matrix, let $\eig_1(\bm
{A})\le
\cdots\le\eig_p(\bm{A})$ stand for its ordered eigenvalues.
We consider the following norms on $p\times p$ matrices
\begin{eqnarray*}
&\|\bmA\|_r =\left(\sum_{i=1}^p |a_{ij}|^r\right)^{1/r}, \; 1\le
r<\iy,
\quad\|\bmA\|_\infty= \mathop{\max}_{i,j}|a_{ij}|, &\nonumber\\
&\|\bmA\|_{(r,s)} = \mbox{sup}\{\|\bmA\bm{x}\|_s:\|\bm{x}\|_r = 1\},
&\nonumber
\end{eqnarray*}
by respectively viewing $\bm{A}$ as a vector in $\RR^{p^2}$ and an
operator from $(\RR^p,\|\cdot\|_r)$ to $(\RR^p, \|\cdot\|_s)$, where
$1\le
r,s\le\iy$.
This gives
\begin{eqnarray*}
&\|\bmA\|_{(1,1)} = \mathop{\max}_{j} \sum_i |a_{ij}|, \quad\|\bmA
\|
_{(\infty,\infty)} = \mathop{\max}_{i} \sum_j |a_{ij}| &\nonumber
\\
&\|\bmA\|_{(2,2)} = \{\max(\eig_i(\bmA^T\bmA):1\le i\le p)\}^{1/2},
\nonumber
\end{eqnarray*}
and that for symmetric matrices, $\|\bmA\|_{(2,2)}= \max\{|\eig
_i(\bmA
)|:1\le i\le p\}$, and $\|\bmA\|_{(1,1)} = \|\bmA\|_{(\infty,\infty)}$.
The norm $\|\cdot\|_{(r,r)}$ will be referred to as the $L_r$-operator
norm. For two matrices $\bm{A}$ and $\bm{B}$, we say that $\bm{A}\ge
\bm
{B}$ (respectively, $\bm{A}> \bm{B}$) if $\bm{A}-\bm{B}$ is nonnegative
definite (respectively, positive definite). Thus $\bmA> \bm{0}$ for a
positive definite matrix $\bmA$, where $ \bm{0}$ stands for the zero
matrix. The identity matrix of order $p$ will be denoted by $\bm{I}_p$.
A vector of $1$'s is denoted by $\bm{1}$.

Sets are denoted by non-bold uppercase English letters. For a set $T$,
we denote the cardinality, that is, the number of elements in $T$, by
$\#T$. We denote the submatrix of the matrix $\bmA$ induced by the set
$T \subset\{1,2,\ldots,p\}$ by $\bmA_T$, i.e., $\bmA_T=\defMatrix
{a_{ij}:\, i,j\in T}$. By $\bmA_T^{-1}$, we mean the inverse $(\bmA
_T)^{-1}$ of the submatrix $\bmA_T$. For a $p\times p$ matrix $\bmA=
(\!(a_{ij})\!)$, let $(\bmA_T)^0 = (\!(a^*_{ij})\!)$ denote a
$p$-dimensional matrix such that $a^*_{ij} = a_{ij}$ for $(i,j) \in T
\times T$, and $0$ otherwise. Also we denote the ``banded'' version of
$\bmA$ by $B_k(\bmA) = (\!(a_{ij}\Ind\{|i-j|\leq k\})\!)$ corresponding
to banding parameter $k$, $k < p$.

\subsection{Preliminaries on graph theory}
Now we discuss some preliminaries on graph theory and undirected
graphical models needed to describe our results. Further details are
available in \cite{lauritzen1996graphical, letac2007wishart}.

An undirected graph $G = (V,E)$ consists of a non-empty vertex set $V =
\{1,2,\ldots,p\}$ along with an edge-set $E \subseteq\{(i,j)\in V
\times V:\, i<j\}$. Two vertices $v, v' \in V$ are said to be adjacent
if there is an edge between $v$ and $v'$. A graph is called complete if
all the vertices are adjacent to each other. A graph $G' = (V',E')$ is
a subgraph of $G = (V,E)$, denoted by $G' \subseteq G$ if $V' \subseteq
V$ and $E' \subseteq E$. For a subgraph $G' \subseteq G$, if $E' = (V'
\times V') \cap E$, then $G'$ is called an induced subgraph of $G$. For
a subset $V' \subseteq V$, we denote the subgraph $G_{V'} = (V', (V'
\times V') \cap E)$ to be the graph induced by $V'$. We shall only
consider induced subgraphs henceforth when we refer to subgraphs of a
graph. A subset $V' \subseteq V$ is said to be a clique of $G$ if the
subgraph $G_{V'}$ is a maximal complete subgraph of $G$, that is,
$G_{V'}$ is not contained in any other complete subgraph
of~$G$.\looseness=-1

A path in a graph is a finite collection of adjacent edges. If $G_1,
G_2$ and $G_3$ are subgraphs of $G$, then $G_3$ is said to separate
$G_1$ and $G_2$ if every path from $j \in G_1$ to $k \in G_2$ contains
a vertex in $G_3$. A graph $G$ decomposes in to disjoint subgraphs
$G_1, G_2$ and $G_3$ if (i) $G_1 \cup G_2 \cup G_3 = G$, (ii) $G_3$ is
complete, and (iii) $G_3$ separates $G_1$ and $G_2$. The decomposition
of a graph is proper if neither $G_1$ nor $G_2$ is empty. A graph is
decomposable if it is complete, or if there exists a proper
decomposition $(G_1, G_2,G_3)$ in decomposable subgraphs induced by the
vertices in $G_1 \cup G_3$ and $G_2 \cup G_3$. This is a recursive
definition which ultimately gives a sequence of cliques and separators
of the graph. One of the most important results in this context of
decomposability of a graph is that of perfect ordering of cliques. A
set of cliques $\mathcal{C}_G = \{C_1,C_2,\ldots,C_r\}$ is said to be
in perfect order, if the following holds: For
%
\begin{equation}
\begin{split}
H_1 = R_1 = C_1, \quad& H_j = C_1 \cup\cdots\cup C_j,\\
R_j = C_j \diagdown H_{j-1}, \quad& S_j = H_{j-1} \cap C_j,\,
j=2,\ldots,r,
\end{split}
\end{equation}
$\mathcal{S} = \{S_j,\, j=2,\ldots,p\}$ is the set of minimal
separators in $G$. The sets $H_j, R_j$ and $S_j$ are termed as
histories, residuals and separators of the sequence respectively.
For a decomposable graph, a perfect order of the cliques always exists.
Figure~\ref{fig:decomfig} illustrates a decomposable and a
non-decomposable graph.

\begin{figure}[t]
\includegraphics[width=4.8in,height=1.92in]{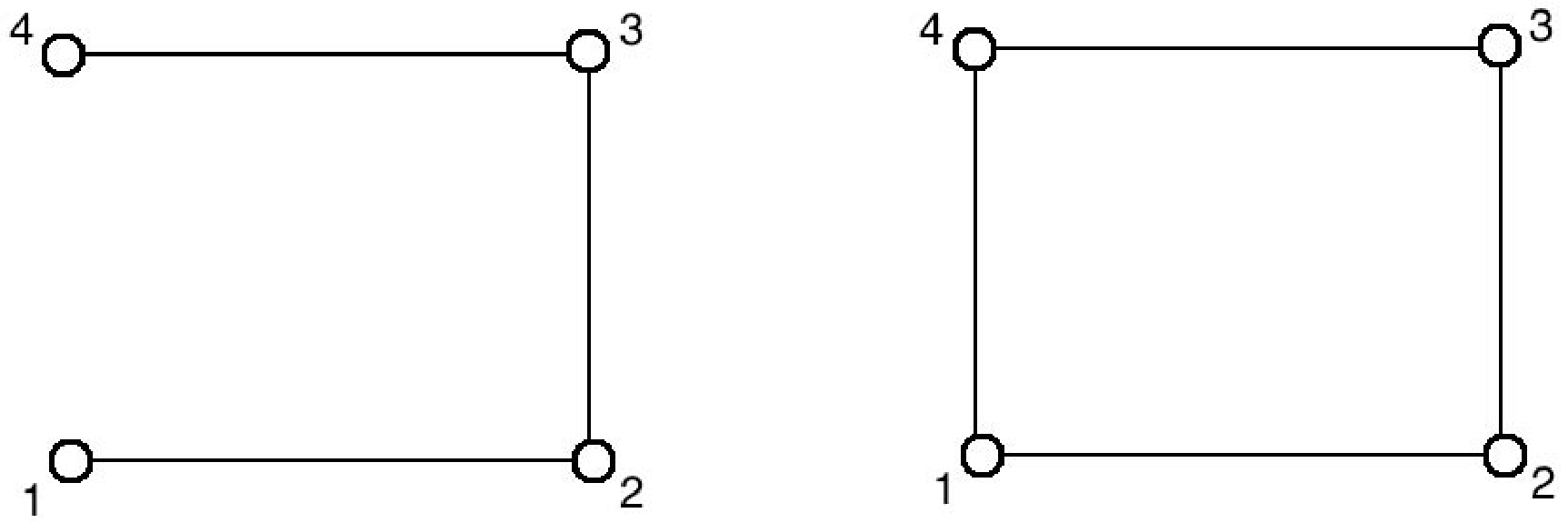}
\caption{[Left] An example of a decomposable graph with vertex set $V =
\{1,2,3,4\}$. $\{1,2\},\allowbreak \{2,3\}$ and $\{3,4\}$ are the cliques, whereas
$\{2\}$ and $\{3\}$ are the separators. [Right]~A~non-decomposable
graph with the same vertex set $V = \{1,2,3,4\}$. There are four
cliques $\{1,2\},\allowbreak \{1,4\},\{2,3\},\{3,4\}$, but they cannot be arranged
in a perfect order, violating the decomposability condition.}
\label{fig:decomfig}
\end{figure}

\subsection{Undirected Gaussian graphical models}

An undirected graph $G$ equipped with a probability distribution $P$
such that the vertex set $V = \{1,\ldots,p\}$ of $G$ corresponds to a
$p$-dimensional random variable $\bX= (X_1,X_2,\ldots,X_p)^T \sim P$,
and for any pair $(i,j)\not\in E$, $i\ne j$, the random variables
$X_i$ and $X_j$ are conditionally independent given all $X_k$, $k\ne
i,j$, is
referred to as an undirected graphical model $(G,P)$. The conditional
independence property is also called the Markov property of $\bX$ with
respect to the graph $G$. If $\bX$ has a multivariate normal
distribution, the graphical model is called a Gaussian graphical model
(GGM). If $\bX$ has mean $\bm{0}$ (without loss of generality) and
positive definite covariance matrix $\bSigma$, then $\bX$ has Markov
property with respect to $G$ if and only if $\omega_{ij} = 0$
for any pair $(i,j) \notin E$, $i\ne j$, where $\omega_{ij}$ is the
$(i,j)$th entry of
$\bOmega=\bSigma^{-1}$; see \cite{lauritzen1996graphical} for a proof.
Thus, for a GGM, absence of an edge between any two vertices is
equivalent to a zero entry in $\bOmega$. Figure~\ref{fig:figure1}
illustrates the connection between a banded precision matrix and the
corresponding graphical model.
In general in a graphical model for a non-Gaussian vector with finite
second moment, for any pair $(i,j) \notin E$, $i\ne j$, we have $\omega
_{ij} = 0$, but $\omega_{ij}=0$ does not imply that $(i,j)\not\in E$.

\begin{figure}[t]
\includegraphics[width=4.8in,height=1.92in]{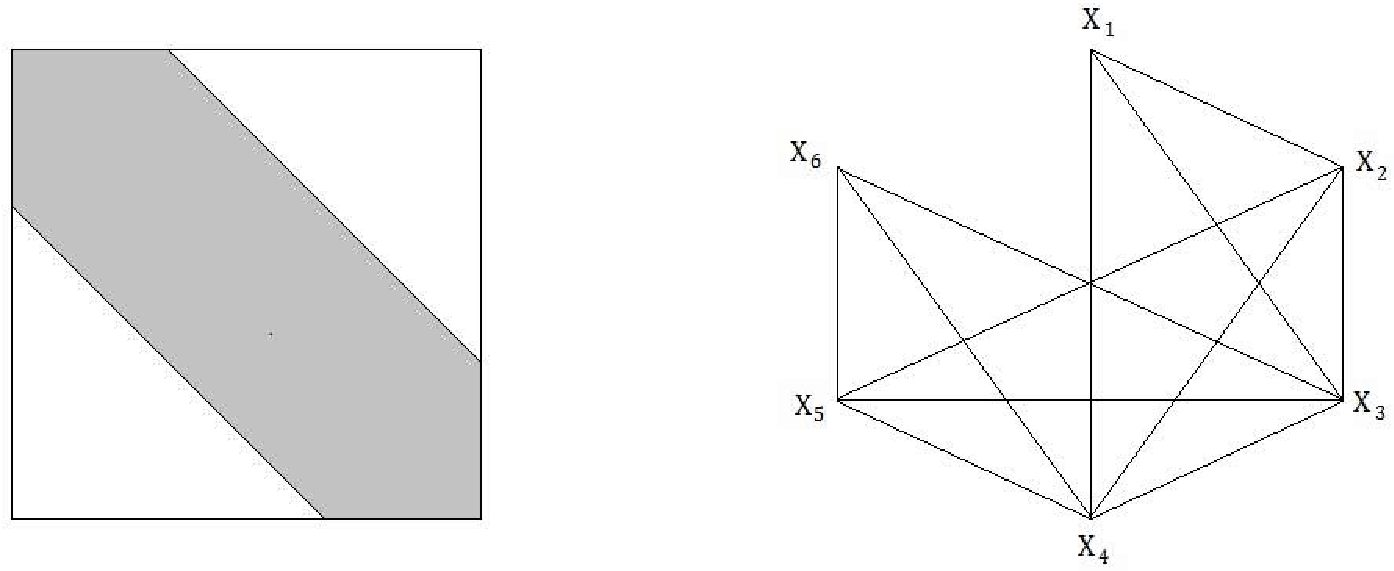}
\caption{[Left] Structure of a banded precision matrix with shaded
non-zero entries. [Right]~The~graphical model corresponding to a banded
precision matrix of dimension 6 and banding parameter 3.}
\label{fig:figure1}
\end{figure}

Let us denote the linear space of $p$-dimensional symmetric matrices by
$\mathcal{M}_p$ and $\mathcal{M}_p^+ \subset\mc{M}_p$ to be the cone
of positive definite matrices of order $p$. Following the notation in
\cite{letac2007wishart}, we can restrict the canonical parameter
$\bOmega$ in $\mathcal{P}_G$, where $\mathcal{P}_G$ is the cone of
positive definite symmetric matrices of order $p$ having zero entry
corresponding to each pair $(i,j)\not\in E$, $i\ne j$, that is,
%
\begin{equation}
\mathcal{P}_G = \{\bOmega= (\!(\omega_{ij})\!) \in\mathcal
{M}_p^+:\;
\omega_{ij} = 0 \ \mathrm{whenever}\,(i,j) \notin E, i\ne j\}.
\end{equation}
The linear space of symmetric incomplete matrices $\bmA= (\!(a_{ij})\!
)$ with missing entries $a_{ij},\, (i,j) \notin E$, will be denoted by
$\mathcal{I}_G$. Any such matrix $\bmA\in\mathcal{I}_G$ is said to be
partially positive definite over $G$ if for every clique $C \in\mc
{C}_G$, the corresponding submatrix $\bmA_{C}$ is positive definite. We
denote the cone of partially positive definite matrices by
%
\begin{equation}
\mathcal{Q}_G = \{\bmB\in\mathcal{I}_G:\; \bmB_{C_i} > \bm{0}, C_i
\in\mc{C}_G \}.
\end{equation}

Gr{\"o}ne et al. \cite{grone1984positive} proved that there is a
bijection between the spaces $\mathcal{P}_G$ and $\mathcal{Q}_G$ for
decomposable graphs $G$. Note that, for $(i,j) \notin E$, the
corresponding entries in $\bSigma$ are not free parameters of the
Gaussian model (see \cite{rajaratnam2008flexible}). For decomposable
graphs, Gr{\"o}ne et al. \cite{grone1984positive} also showed that for
every $\bSigma\in\mc{Q}_G$ there exists a unique positive definite
matrix $\bSigma^*$ such that $\bSigma^*_{ij} = \bSigma_{ij}$ for $(i,j)
\in E$ such that $\bSigma^*$ corresponds to the covariance matrix of a
Gaussian distribution which is Markov with respect to the underlying
graphical structure. Thus when $G$ is decomposable, the parameter space
for $\bSigma$ can be defined by the set of incomplete matrices which
are partially positive definite, that is, we can restrict the
covariance matrix $\bSigma$ in $\mc{Q}_G \subset\mc{I}_G$. More
precisely, we can define the parameter space for the GGM as the space
of incomplete matrices $\bSigma\in\mc{Q}_G$ such that $\bSigma=
\kappa(\bOmega^{-1}),\, \bOmega\in\mathcal{P}_G$, where $\kappa:
\mathcal{M}_p \rightarrow\mathcal{I}_G$ is the projection of
$\mathcal
{M}_p$ into~$\mathcal{I}_G$.\looseness=1

To give a simple example to illustrate the parameter spaces, consider
the decomposable graph as in the left side of Figure \ref
{fig:decomfig}, such that the underlying random variables $\bX=
(X_1,\ldots,X_4)$ corresponding to the vertices of respective indices
follow an autoregressive process of order $1$, with covariance between
$X_i$ and $X_j$ given by $0.7^{|i-j|}$, for $i = 1,\ldots,4,\, j =
1,\ldots,4$. The corresponding precision matrix $\bOmega= (\!(\omega
_{ij})\!)$ is given by
\begin{equation*}
\bOmega=
\begin{pmatrix}
1.961 & -1.373 & 0 & 0 \\
-1.373 & 2.922 & -1.373 & 0 \\
0 & -1.373 & 2.922 & -1.373 \\
0 & 0 & -1.373 & 1.961
\end{pmatrix}
.
\end{equation*}
Note that in the above example, $\omega_{ij} = 0$ whenever $(i,j)
\notin E$, $i\ne j$, $E$ being the edge-set in the underlying graph.
For the space of covariance matrices, it is enough to consider the
incomplete matrix $\bSigma= (\!(\sigma_{ij})\!)$, where $\sigma_{ij} =
0.7 ^ {|i - j|}$ for $(i,j) \in E$ and $\sigma_{ij}$ is a missing entry
otherwise. Clearly the submatrices of $\bSigma$ corresponding to the
cliques of the underlying graph are positive definite. This incomplete
matrix can be uniquely extended to the positive definite matrix
$\bSigma
^* = (\!(\sigma^*_{ij})\!)$, where $\sigma^*_{ij} = 0.7 ^ {|i-j|}$ for
all $i,j \in\{1,\ldots,4\}$.

We shall work only with decomposable Gaussian graphical models in this paper.

\section{Model assumption and prior specification}
\label{sec:modelandprior}
Let $\bX_1,\bX_2,\ldots,\bX_n$ be independent and identically
distributed (i.i.d.) random $p$-vectors with mean $\mathbf{0}$ and
covariance matrix $\bSigma$. Write $\bX_i = (X_{i1},X_{i2},\ldots,X_{ip})^T$,
and assume that the $\bX_i$, $i=1,\ldots,n$, are multivariate Gaussian.
Consistent estimators for the covariance matrix were obtained in \cite
{bickel2008regularized} by banding the sample covariance matrix,
assuming a certain sparsity structure on the true covariance. Our aim
is to obtain convergence rates of the graphical MLE and Bayes
estimators of the precision matrix $\bOmega= \bSigma^{-1}$ under the
condition $n^{-1}\log p \rightarrow0$ where $\bOmega$ ranges over some
fairly natural families. For a given positive sequence $\ga
(k)\downarrow0$, we consider the class of positive definite symmetric
matrices $\bOmega= (\!(\omega_{ij})\!)$ as
%
\begin{equation}
\label{matrixclass1}
\begin{split}
\mathcal{U}(\varepsilon_0,\gamma) &= \left\{\bOmega: \mathop{\max}_{j}
\sum_i\{|\omega_{ij}|: |i-j| > k\} \leq\gamma(k)\, \mathrm{for\,
all}\, k>0, \right. \\
&\qquad\left. \vphantom{\sum_i} 0< \varepsilon_0 \leq\eig
_1(\bOmega)
\leq\eig_p(\bOmega) \leq\varepsilon_0^{-1} < \infty\right\}.
\end{split}
\end{equation}
The sequence $\gamma(k)$ which bounds $\|\bOmega-B_k(\bOmega)\|
_{(\iy
,\iy)}=\mathop{\max}_{j} \sum_i\{|\omega_{ij}|:\break  |i-j| > k\}$ has been
kept flexible so as to include a number of matrix classes.
\begin{enumerate}
\item Exact banding: $\gamma(k) = 0$ for all $k \geq k_0$, which means
that the true precision matrix is banded, with banding parameter $k_0$.
For instance, any autoregressive process has such a form of precision matrix.
\item Exponential decay: $\gamma(k) = e^{-ck}$. For instance, any
moving average process has such a form of precision matrix.
\item Polynomial decay: $\gamma(k) = \gamma k^{-\alpha}$, $\alpha> 0$.
This class of matrices was considered in \cite{bickel2008regularized,
cai2012adaptive}.
\end{enumerate}
We shall estimate the precision matrix which belongs $\mathcal
{U}(\varepsilon_0,\gamma)$ for some $\ve_0$ and $\ga(\cdot)$ by
restricting the prior on a sieve of banded matrices.
A banding structure in the precision matrix can be induced by a
Gaussian graphical model. Since $\omega_{ij} = 0$ implies that the
components $X_i$ and $X_j$ of $\bX$ are conditionally independent given
the others, we can thus define a Gaussian graphical model $G =(V,E)$,
where $V=\{1,\ldots,p\}$ indexing the $p$ components $X_1,X_2,\ldots,X_p$,
and $E$ is the corresponding edge set defined by $E=\{(i,j):0<|i-j|\leq
k\}$, and $k$ is the size of the band. This describes a parameter space
for precision matrices consisting of $k$-banded matrices, and can be
used for the maximum likelihood or the Bayesian approach, where for the
latter, a prior distribution on these matrices must be specified.

It is not difficult to check that $G$ is an undirected, decomposable
graphical model for which a perfect order of cliques exist, given by
$\mathcal{C} = \{C_1,C_2,\ldots,C_{p-k}\}$, $C_j = \{j,j+1,\ldots
,j+k\}
$, $j=1,2,\ldots,p-k$. The corresponding separators are given by
$\mathcal{S}=\{S_2,S_3,\ldots,S_{p-k}\}$, $S_j = \{j,j+1,\ldots
,j+k-1\}
$, $j=2,3,\ldots,p-k$. The choice of the perfect set of cliques is not
unique, but the estimator for the precision matrix $\bOmega$ under all
choices of the order remains the same. We shall work with the
$G$-Wishart distribution $W_G(\delta,\bmD)$ as a conjugate prior for
$\bOmega\in\mc{P}_G$. The prior density, derived by \cite
{roverato2000cholesky}, is given by
%
\begin{equation}
\label{eqn:Gdensity}
p(\bOmega|G) = (I_G(\delta,\bmD))^{-1}(\det(\bOmega))^{(\delta
-2)/2}\,
\mbox{exp}\left[-\frac{1}{2}\mbox{tr}(\bmD\bOmega)\right],
\end{equation}
where $\bmD$ is a symmetric positive definite matrix and
%
\begin{equation}
\label{eqn:Gnorm}
I_G(\delta,\bmD) = \int_{\bOmega\in\mathcal{P}_G} (\det(\bOmega
))^{(\delta-2)/2}\,\mbox{exp}\left[-\frac{1}{2}\mbox{tr}(\bmD
\bOmega
)\right]\, d\bOmega
\end{equation}
is the normalizing constant, which is finite for $\delta> 2$.
Letac and Massam \cite{letac2007wishart} introduced a more general
family of conjugate priors, called the $W_{P_G}$-Wishart family, as a
prior distribution for $\bOmega$. The $W_{P_G}$-Wishart distribution
$W_{P_G}(\balpha,\bbeta,\bmD)$ has three set of parameters $\balpha$,
$\bbeta$ and $\bmD$,
where $\balpha$ and $\bbeta$ are suitable functions defined on the
cliques and separators of the graph respectively, and $\bmD$ is a
scaling matrix.
The $G$-Wishart distribution $W_G(\de,\bmD)$ is a special case of the
$W_{P_G}$-Wishart family where
%
\begin{equation}
\begin{split}
\alpha_i &= -\frac{\delta+ \# C_i -1 }{2}, \, i = 1,2,\ldots,p-k, \\
\beta_i &= -\frac{\delta+ \# S_i -1 }{2},\, i=2,3,\ldots,p-k.
\end{split}
\end{equation}
If the prior distribution on $\frac{1}{2}\bOmega$ is $W_{P_G}(\balpha
,\bbeta,\bmD)$, then the posterior distribution of $\frac
{1}{2}\bOmega$
given the sample covariance $\bmS= n^{-1}\sum_{i=1}^{n}\bX_i\bX_i^T$
is given by $W_{P_G}(\balpha-\frac{n}{2}\bm{1},\bbeta-\frac
{n}{2}\bm
{1},\bmD+ \kappa(n\bmS))$. For the $G$-Wishart distribution in our
case, $\#C_i = k+1$ for all $i = 1,2,\ldots,p-k $, and $\#S_j =k$ for
all $j = 2,3,\ldots,p-k$. Thus
%
\begin{equation}
\begin{split}
\alpha_i &= -\frac{\delta+ k}{2}, \, i = 1,2,\ldots,p-k, \\
\beta_j &= -\frac{\delta+ k -1 }{2},\, j=2,3,\ldots,p-k.
\end{split}
\end{equation}

The posterior mean of $\bOmega$ was derived in \cite
{rajaratnam2008flexible}, and is given by
%
\begin{equation}
\begin{split}
\E(\bOmega|\bmS) &= -2\left[\sum_{j=1}^{p-k}(\alpha_j-\frac
{n}{2})\left
((\bmD+\kappa(n\bmS))_{C_j}^{-1}\right)^0 \right. \\
&\qquad\qquad\left. -\sum_{j=2}^{p-k}(\beta_j-\frac{n}{2})\left
((\bmD
+\kappa(n\bmS))_{S_j}^{-1}\right)^0\right].
\end{split}
\end{equation}
Taking $\bmD= \bmI_p$, the $p$ dimensional indicator matrix, and
plugging in the values of $\balpha$ and $\bbeta$, we get the posterior
mean with respect to the $G$-Wishart prior $W_G(\de,\bm{I}_p)$ as,
%
\begin{equation}
\label{eqn:Bayesestimator}
\begin{split}
\E(\bOmega|\bmS) &=\f{\delta+k+n}{n}\left[\sum_{j=1}^{p-k}\left
((n^{-1}\bmI_{k+1}+\bmS_{C_j})^{-1}\right)^0 \right. \\
&\qquad\left. -\sum_{j=2}^{p-k}\left((n^{-1}\bmI_{k}+\bmS
_{S_j})^{-1}\right)^0 \right] + n^{-1}\sum_{j=2}^{p-k}\left
((n^{-1}\bmI
_{k}+\bmS_{S_j})^{-1}\right)^0.
\end{split}
\end{equation}
For a sample of size $n$ from a $p$-dimensional Gaussian distribution
with mean $\bm{0}$ and precision matrix $\bOmega$, we consider the
following two loss functions:
%
\begin{equation}
\begin{split}
\mbox{Stein's loss: }& L_1(\widehat{\bOmega},\bOmega) = \frac
{1}{2}\tr
(\widehat{\bOmega}\bOmega^{-1}) - \log|\widehat{\bOmega}\bOmega
^{-1}| -
p, \\
\mbox{Squared-error loss: }& L_2(\widehat{\bOmega},\bOmega) = \tr
(\widehat{\bOmega} - \bOmega)^2,
\end{split}
\end{equation}
for an arbitrary estimator $\widehat{\bOmega}$ of $\bOmega$. The Bayes
estimators corresponding to the above two loss\vadjust{\goodbreak} functions were derived
in \cite{rajaratnam2008flexible}. Under the $G$-Wishart prior $W_G(\de
,\bm{I}_p)$, the Bayes estimator $\widehat{\bOmega}_{L_1}^{\mathrm{B}}$
corresponding to Stein's loss function is given by
%
\begin{equation}
\f{\delta+n-2}{n}\left[\sum_{j=1}^{p-k}\left((n^{-1}\bmI
_{k+1}+\bmS
_{C_j})^{-1}\right)^0 -\sum_{j=2}^{p-k}\left((n^{-1}\bmI_{k}+\bmS
_{S_j})^{-1}\right)^0 \right].
\end{equation}
For the squared-error loss function, the corresponding Bayes estimator
is clearly the posterior mean of $\bOmega$ given in (\ref
{eqn:Bayesestimator}). We denote this estimator by $\widehat{\bOmega
}_{L_2}^{\mathrm{B}}$. Some other loss functions for estimation of
$\bOmega$ have also been considered in the literature; see \cite
{yangberger1994}.

The graphical MLE for $\bOmega$ under the graphical model with banding
parameter $k$ is given by (see \cite{lauritzen1996graphical}),
%
\begin{equation}
\label{graphical mle}
\widehat{\bOmega}^{\mathrm{M}} = \sum_{j=1}^{p-k}(\bmS
_{C_j}^{-1})^0 -
\sum_{j=2}^{p-k}(\bmS_{S_j}^{-1})^0.
\end{equation}

\section{Main results}
\label{sec:mainresults}
In this section, we determine the convergence rate of the posterior
distribution of the precision matrix. The following theorem describes
the behavior of the entire posterior distribution.

\begin{theorem}
\label{theorem:Posteriorconvergence}
Let $\bX_1,\ldots,\bX_n$ be random samples from a $p$-dimensional Gaussian
distribution with mean zero and true precision matrix $\bOmega_0\in
\mathcal{U}(\ve_0,\gamma)$ for some $\ve_0>0$ and $\ga(\cdot)$
such that
$k^{3/2}\gamma(k)\to0$ as $k\to\infty$. Suppose that $\bOmega$ is
given the $G$-Wishart prior $W_G(\de, \bm{I}_p)$, where the graph $G$
has banding of order $k$. Then posterior distribution of the precision
matrix $\bOmega$ satisfies
%
\begin{equation}
\E_0\left[\P\left\{\|\bOmega- \bOmega_0\|_{(\infty,\infty)} >
M\epsilon
_{n,k} | \bX_1,\ldots,\bX_n\right\}\right] \rightarrow0
\end{equation}
for $\epsilon_{n,k} = k^{5/2}(n^{-1} \log p)^{1/2} + k^{3/2}\gamma(k)$
and a sufficiently large constant $M>0$.

In particular, the posterior distribution is consistent in the $L_\iy
$-operator norm if $k\to\iy$ such that $k^5 n^{-1}\log p \to0$.
\end{theorem}

An important step towards the proof of the above result is to find the
convergence rate of the graphical MLE, which is also of independent
interest. For high-dimensional situations, even when the sample
covariance matrix is singular, the graphical MLE will be positive
definite if the number of elements in the cliques of the corresponding
graphical model is less than the sample size.

Convergence results for banded empirical covariance (or precision)
matrix or estimators based on thresholding approaches are typically
given in terms of the $L_2$-operator norm in the literature. We however
use the stronger $L_\iy$-operator norm (or equivalently, $L_1$-operator
norm), so the implication of a convergence rate in our theorems is stronger.

\begin{proposition}
\label{theorem:mleconvergence}
Let $\bX_1,\ldots,\bX_n$ be random samples from a $p$-dimensional Gaussian
distribution with mean zero and precision matrix $\bOmega_0\in
\mathcal
{U}(\ep_0,\gamma)$ for some $\ep_0>0$ and $\ga(\cdot)$ such that
$k^{3/2}\gamma(k)\to0$ as $k\to\infty$. Then\vadjust{\goodbreak} the graphical MLE
$\widehat{\bOmega}^{\mathrm{M}}$ of $\bOmega$, corresponding to the
Gaussian graphical model with banding parameter $k$, has convergence
rate given by
%
\begin{equation}
\|\widehat{\bOmega}^{\mathrm{M}} - \bOmega_0\|_{(\infty,\infty)} =
O_P\left(k^{5/2}(n^{-1} \log p)^{1/2} + k^{3/2}\gamma(k)\right).
\end{equation}
In particular, $\widehat{\bOmega}^{\mathrm{M}}$ is consistent in the
$L_\iy$-operator norm if $k\to\iy$ such that $k^5 n^{-1}\log p \to0$.
\end{proposition}

The proof will use the explicit form of the graphical MLE and proceed
by bounding the mean squared error of each component and using
relations between matrix norms. However, as the graphical MLE involves
$(k+1)(p-k/2)$ many terms, a naive approach will lead to a factor $p$
in the estimate, which will not be able to establish a convergence rate
in the truly high dimensional situations $p\gg n$. We overcome this
obstacle by looking more carefully at the structure of the graphical
MLE, and note that for any $i$, the number of terms in \eqnref
{graphical mle} which have non-zero $i$th row is only at most
$(2k+1)\ll p$. This along with the description of $L_\iy$-operator norm
in terms of row sums give rise to a much smaller factor than $p$.

Now we treat the Bayes estimators. Consider the $G$-Wishart prior
$W_G(\de, \bm{I}_p)$ for $\bOmega$, where the graph $G$ has banding of
order $k$ and $\de$ is a positive integer.\vspace*{1pt}
The following result bounds the difference between $\widehat{\bOmega
}^{\mathrm{M}}$ and the estimators $\widehat{\bOmega}_{L_1}^{\mathrm
{B}}$ and $\widehat{\bOmega}_{L_2}^{\mathrm{B}}$.

\begin{lemma}
\label{lemma:Bayesmle}
Assume the conditions of Proposition \ref{theorem:mleconvergence} and
suppose that $\bOmega$ is given the $G$-Wishart prior $W_G(\de, \bm
{I}_p)$, where the graph $G$ has banding of order $k$. Then
$\|\widehat{\bOmega}_{L_1}^{\mathrm{B}} - \widehat{\bOmega
}^{\mathrm
{M}}\|_{(\infty,\infty)} = O_P(k^{2}/n)$,
$\|\widehat{\bOmega}_{L_2}^{\mathrm{B}} - \widehat{\bOmega
}^{\mathrm
{M}}\|_{(\infty,\infty)} = O_P(k^{5/2}/n)$.
\end{lemma}

The proof of the above lemma is given in the Section \ref{sec:proofs}.
Proposition \ref{theorem:mleconvergence} and Lemma \ref{lemma:Bayesmle}
together lead to the following result for the convergence rate of the
Bayes estimators under the $G$ in the $L_{\infty}$-operator norm.

\begin{proposition}
\label{prop:Bayesconvergence}
In the setting of Lemma~\ref{lemma:Bayesmle}, for $\widehat{\bOmega
}^{\mathrm{B}}$ either $\widehat{\bOmega}_{L_1}^{\mathrm{B}}$ or
$\widehat{\bOmega}_{L_2}^{\mathrm{B}}$, we have
%
\begin{equation}
\|\widehat{\bOmega}^{\mathrm{B}} - \bOmega_0\|_{(\infty,\infty)} =
O_P\left(k^{5/2}(n^{-1} \log p)^{1/2} + k^{3/2}\gamma(k)\right).
\end{equation}
In particular, the Bayes estimators $\widehat{\bOmega}_{L_1}^{\mathrm
{B}}$\vspace*{1pt} and $\widehat{\bOmega}_{L_2}^{\mathrm{B}}$ are consistent in the
$L_\iy$-operator norm if $k\to\iy$ such that $k^5 n^{-1}\log p \to0$.
\end{proposition}

\noindent{\bf Remarks on the convergence rates.}\ \
Observe that the convergence rates of the graphical MLE, the Bayes
estimators and the posterior distribution obtained above are the same.
The obtained rates can be optimized by choosing $k$ appropriately as in
a bias-variance trade-off.
The fastest possible rates obtained from the theorems may be summarized
for the different decay rates of $\ga(k)$ as follows: If the true
precision matrix is banded with banding parameter $k_0$, then the
optimal rate of convergence $n^{-1/2}(\log p)^{1/2}$ is obtained by
choosing any fixed $k\ge k_0$. When $\ga(k)$ decays exponentially, the
rate of convergence $n^{-1/2}(\log p)^{1/2}(\log n)^{5/2}$ can be
obtained by choosing $k$ approximately proportional to $\log n$ with
some sufficiently large constant of proportionality. If $\ga(k)$ decays
polynomially with index $\al>3/2$ as in \cite{bickel2008regularized},
we get the convergence rate of $(n^{-1}\,\log\,p)^{(2\alpha-
3)/(4\alpha+4)}$ corresponding $k \asymp(n/\log\, p)^{1/(2\alpha+ 2)}$.

It is to be noted that we have not assumed that the true structure of
the precision matrix arises from a graphical model. The graphical model
is a convenient tool to generate useful estimators through the maximum
likelihood and Bayesian approach, but the graphical model itself may be
a misspecified model.
Further, it can be inspected from the proof of the theorems that the
Gaussianity assumption on true distribution of the observations is not
essential, although the graphical model assumes Gaussianity to generate
estimators. The Gaussianity assumption is used to control certain
probabilities by applying the probability inequality Lemma~A.3 of \cite
{bickel2008regularized}. However, it was also observed in
\cite{bickel2008regularized} that one only requires bounds on the
moment generating function of $X_i^2$, $i=1,\ldots,p$. In particular, any
thinner tailed distribution, such as one with a bounded support, will
allow the arguments to go through.

\subsection{Estimation using a reference prior}
A reference prior for the covariance matrix $\bSigma$, obtained in
\cite
{rajaratnam2008flexible}, can also be used to induce a prior on
$\bOmega
$. This corresponds to an improper $W_{P_G}(\balpha,\bbeta,\bm{0})$
distribution for $\half\bOmega$ with
%
\begin{equation}
\begin{split}
\alpha_i &= 0, \qquad i =1,2,\ldots,r, \\
\beta_2 &= \frac{1}{2}(c_1+c_2)-s_2,\, \beta_j = \frac{1}{2}(c_j-s_j),
\qquad j=2,3,\ldots,r.
\end{split}
\end{equation}
By Corollary 4.1 in \cite{rajaratnam2008flexible}, the posterior mean
$\widehat{\bOmega}^{\mathrm{R}}$ of the precision matrix is given by
%
\begin{equation}
\begin{split}
\widehat{\bOmega}^{\mathrm{R}} ={}& \sum_{j=1}^{r}(\bmS
_{C_j}^{-1})^0 - \{
1 - n^{-1}(c_1+c_2-2s_2)\}(\bmS_{S_2}^{-1})^0 \\
&{} - \sum_{j=3}^{r}\{1-n^{-1}(c_j-s_j)\}(\bmS_{S_j}^{-1})^0.
\end{split}
\end{equation}
Similar to the conclusion of Lemma \ref{lemma:Bayesmle}, using the
reference prior, the $L_{\infty}$-operator norm of the difference
between the Bayes estimator $\widehat{\bOmega}^{\mathrm{R}}$ and the
graphical MLE $\widehat{\bOmega}^{\mathrm{M}}$ satisfies
%
\begin{equation}
\label{eqn:refprior}
\|\widehat{\bOmega}^{\mathrm{R}} - \widehat{\bOmega}^{\mathrm
{M}}\|
_{(\infty,\infty)} = O_P\left(k^2/n\right).
\end{equation}

A sketch of the proof is given in Section \ref{sec:proofs}.

\section{Estimation of banding parameter}
\label{sec:bandestim}
In this section, we propose a method of selecting the banding parameter
$k$ of the graphical model using the marginal posterior probabilities
of the graph induced by banding $k$, $k=1,2,\ldots$. For the $G$-Wishart
prior $W_G(\delta,\bmD)$ for $\bOmega$, the posterior is given by
$W_G(\delta+n,\bmD+n\bmS)$. We can get the\vadjust{\goodbreak} marginal likelihood for
$G$ as
%
\begin{equation}
p(\bX|G) = (2\pi)^{-np/2}\frac{I_G(\delta+n,\bmD+n\bmS
)}{I_G(\delta,\bmD)},
\end{equation}
where $I_G(\delta,\bmD)$ is the normalizing constant in the density
(\ref{eqn:Gdensity}) of $W_G(\delta,\bmD)$, given by (\ref{eqn:Gnorm}).
For a complete graph $G$, the expression for $I_G(\delta,\bmD)$ is
given by
%
\begin{equation}
\label{IGform}
I_G(\delta,\bmD) = \frac{2^{(\delta+p-1)p/2}\pi^{p(p-1)/4}\prod
_{i=0}^{p-1}\Gamma\left(\frac{\delta+p-1-i}{2}\right)}
{(\det(\bmD))^{\frac{\delta+p-1}{2}}};
\end{equation}
see \cite{muirheadaspects}.
Roverato \cite{roverato2000cholesky} showed that for a decomposable
graph $G$,
%
\begin{equation}
\label{IG decomposable}
I_G(\delta,\bmD) = \frac{\prod_{j=1}^{r}I_{{C_j}}(\delta,\bmD
_{C_j})}{\prod_{j=2}^{r}I_{{S_j}}(\delta,\bmD_{S_j})},
\end{equation}
where $\{C_1,\ldots,C_r\}$ and $\{S_2,\ldots,S_r\}$ denote the set of
cliques and separators respectively corresponding to $G$.

In our case, the model which is fit has a banded structure in the
precision matrix. The graphs associated with the banding structure are
linearly indexed by the banding parameter $k$. We denote the graphical
model induced by banding parameter $k$ by $G^k$. Let $\rho_k$ be a
prior on $k$. Then the posterior distribution of $k$ is given by
%
\begin{equation}
\label{posterior for k}
p(k \mid\bX) = \frac{J_{G^k}(\delta,n,\bmD,n\bmS)\rho_k}{\sum
_{k'}J_{G^{k'}}(\delta,n,\bmD,n\bmS)\rho_{k'}},
\end{equation}
where for a graph $G$
%
\begin{equation}
\label{JGform1}
J_{G}(\delta,n,\bmD,n\bmS) = \frac{I_{G}(\delta+n,\bmD+n\bmS
)}{I_{G}(\delta,\bmD)}.
\end{equation}
Let the cliques and separators be respectively denoted by $C_j^k=\{
j,j+1,\ldots,\break j+k\}$, $j=1,\ldots,p-k$, and $S_j^k=\{j,j+1,\ldots
,j+k-1\}$,
$j=2,\ldots,p-k$.
Note that the sub-graphs corresponding to the cliques and separators
are complete, with respective dimensions $k+1$ and $k$, and $r=p-k$.
Therefore (\ref{IGform}) and (\ref{IG decomposable}) together reduces
the expression for $I_{G^k}(\delta,\bmD)$ to
%
\begin{equation}
\frac{\prod_{j=1}^{p-k} 2^{(\delta+k)(k+1)/2} \pi^{k(k+1)/4} \prod
_{i=0}^k \Gamma\left(\frac{\delta+k-i}{2}\right)}
{\prod_{j=2}^{p-k}2^{(\delta+k-1)k/2} \pi^{k(k-1)/4} \prod
_{i=0}^{k-1}\Gamma\left(\frac{\delta+k-1-i}{2}\right)}
{\frac{(\det(\bmD_{S_j}))^{(\delta+k-1)/2}}{(\det(\bmD
_{C_j}))^{(\delta+k)/2}}}.
\end{equation}
Now, with the choice $\bmD= \bmI_p$ used in the prior $W_{G^k}(\de
,\bm
{I})$, (\ref{JGform1}) gives
%
\begin{eqnarray}
J_{G^k}(\delta,n,\bmI_p,n\bmS)
&=& \frac{\prod_{j=1}^{p-k}2^{n(k+1)/2}}{\prod_{j=2}^{p-k}2^{nk/2}}
{\left(\prod_{i=0}^{k}\frac{\Gamma\left(\frac{\delta
+n+i}{2}\right)}
{\Gamma\left(\frac{\delta+i}{2}\right)}\right)}
{\left(\frac{\Gamma\left(\frac{\delta+n+k}{2}\right)}
{\Gamma\left(\frac{\delta+k}{2}\right)}\right)^{p-k-1}} \nn\\
&&{}\times{\frac{\prod_{j=2}^{p-k}(\det((\bmI_p+n\bmS
)_{S_j^k}))^{(\delta+k+n-1)/2}}{\prod_{j=1}^{p-k}
(\det((\bmI_p+n\bmS)_{C_j^k}))^{(\delta+n+k)/2}}}.
\end{eqnarray}
Substituting this expression in (\ref{posterior for k}), we get an
explicit expression for the posterior distribution of $k$.

A natural method of selecting $k$ is to consider the posterior mode. In
the next section, we investigate the performance of the posterior mode
of $G^k$ through a simulation study.

\section{Numerical results}
\label{sec:sim}
We check the performance of the Bayes estimators of the precision
matrix and compare with the graphical MLE and the banded estimator as
proposed in \cite{bickel2008regularized}.

\begin{sidewaystable}
\def\arraystretch{0.98}
\tabcolsep=5.5pt
\caption{Simulation results for AR(1) model based on 100 replications;
figures in parentheses indicate standard errors}\label{table1}
\begin{tabular}{cccccccccccccccc}
\hline
& & \multicolumn{4}{c}{$n = 100$} & & \multicolumn{4}{c}{$n = 200$} &
& \multicolumn{4}{c}{$n = 500$}\\
\cline{3-6} \cline{8-11} \cline{13-16}
$p$ & Norm & MLE & $\bOmega_{L_2}^{\mathrm{B}}$ & $\bOmega
_{L_1}^{\mathrm{B}}$ & Cholesky & & MLE & $\bOmega_{L_2}^{\mathrm{B}}$
& $\bOmega_{L_1}^{\mathrm{B}}$ & Cholesky & & MLE & $\bOmega
_{L_2}^{\mathrm{B}}$ & $\bOmega_{L_1}^{\mathrm{B}}$ & Cholesky \\
\hline
    &                     &         &         &         &         &  &         &         &         &         &  &         &         &         &         \\[-9pt]
    & $L_{\infty,\infty}$ & 1.252   & 1.295   & 1.175   & 1.249   &  & 0.799   & 0.820   & 0.773   & 0.797   &  & 0.477   & 0.485   & 0.470   & 0.477   \\
    &                     & (0.029) & (0.029) & (0.027) & (0.029) &  & (0.018) & (0.018) & (0.017) & (0.018) &  & (0.009) & (0.009) & (0.008) & (0.008) \\
    & $L_{2,2}$           & 1.003   & 1.044   & 0.940   & 0.999   &  & 0.644   & 0.663   & 0.623   & 0.642   &  & 0.374   & 0.381   & 0.368   & 0.374   \\
    &                     & (0.029) & (0.023) & (0.021) & (0.023) &  & (0.016) & (0.016) & (0.015) & (0.016) &  & (0.007) & (0.007) & (0.007) & (0.007) \\
50  & $L_2$               & 2.374   & 2.454   & 2.275   & 2.366   &  & 1.609   & 1.643   & 1.575   & 1.607   &  & 0.976   & 0.986   & 0.968   & 0.975   \\
    &                     & (0.026) & (0.027) & (0.023) & (0.026) &  & (0.017) & (0.017) & (0.016) & (0.016) &  & (0.008) & (0.008) & (0.008) & (0.008) \\
    & $L_{\infty}$        & 0.729   & 0.767   & 0.688   & 0.726   &  & 0.450   & 0.468   & 0.437   & 0.450   &  & 0.272   & 0.278   & 0.268   & 0.272   \\
    &                     & (0.018) & (0.018) & (0.017) & (0.018) &  & (0.011) & (0.011) & (0.010) & (0.011) &  & (0.005) & (0.005) & (0.005) & (0.005) \\[4pt]
    & $L_{\infty,\infty}$ & 1.378   & 1.420   & 1.295   & 1.374   &  & 0.889   & 0.912   & 0.861   & 0.889   &  & 0.525   & 0.534   & 0.516   & 0.525   \\
    &                     & (0.029) & (0.029) & (0.027) & (0.029) &  & (0.018) & (0.018) & (0.017) & (0.018) &  & (0.009) & (0.009) & (0.009) & (0.009) \\
    & $L_{2,2}$           & 1.112   & 1.152   & 1.042   & 1.107   &  & 0.712   & 0.734   & 0.687   & 0.711   &  & 0.408   & 0.416   & 0.401   & 0.408   \\
    &                     & (0.022) & (0.022) & (0.021) & (0.023) &  & (0.015) & (0.015) & (0.015) & (0.015) &  & (0.007) & (0.007) & (0.007) & (0.007) \\
100 & $L_2$               & 3.365   & 3.482   & 3.223   & 3.354   &  & 2.264   & 2.310   & 2.217   & 2.263   &  & 1.383   & 1.397   & 1.371   & 1.382   \\
    &                     & (0.027) & (0.030) & (0.024) & (0.027) &  & (0.016) & (0.017) & (0.015) & (0.016) &  & (0.010) & (0.010) & (0.009) & (0.010) \\
    & $L_{\infty}$        & 0.814   & 0.852   & 0.768   & 0.808   &  & 0.500   & 0.518   & 0.484   & 0.499   &  & 0.298   & 0.305   & 0.293   & 0.298   \\
    &                     & (0.018) & (0.018) & (0.017) & (0.018) &  & (0.010) & (0.011) & (0.010) & (0.010) &  & (0.005) & (0.005) & (0.005) & (0.005) \\[4pt]
    & $L_{\infty,\infty}$ & 1.558   & 1.602   & 1.463   & 1.557   &  & 1.002   & 1.027   & 0.967   & 1.001   &  & 0.582   & 0.593   & 0.572   & 0.582   \\
    &                     & (0.028) & (0.028) & (0.027) & (0.029) &  & (0.019) & (0.019) & (0.019) & (0.019) &  & (0.010) & (0.010) & (0.010) & (0.010) \\
    & $L_{2,2}$           & 1.237   & 1.276   & 1.160   & 1.233   &  & 0.791   & 0.814   & 0.763   & 0.790   &  & 0.453   & 0.463   & 0.445   & 0.453   \\
    &                     & (0.022) & (0.021) & (0.021) & (0.022) &  & (0.015) & (0.015) & (0.015) & (0.015) &  & (0.007) & (0.007) & (0.007) & (0.007) \\
200 & $L_2$               & 4.750   & 4.915   & 4.548   & 4.738   &  & 3.211   & 3.277   & 3.143   & 3.209   &  & 1.971   & 1.987   & 1.955   & 1.970   \\
    &                     & (0.024) & (0.026) & (0.022) & (0.025) &  & (0.017) & (0.018) & (0.017) & (0.017) &  & (0.010) & (0.010) & (0.010) & (0.010) \\
    & $L_{\infty}$        & 0.923   & 0.961   & 0.872   & 0.918   &  & 0.564   & 0.584   & 0.545   & 0.564   &  & 0.327   & 0.334   & 0.321   & 0.326   \\
    &                     & (0.018) & (0.018) & (0.017) & (0.018) &  & (0.011) & (0.011) & (0.011) & (0.011) &  & (0.006) & (0.006) & (0.006) & (0.006) \\[4pt]
    & $L_{\infty,\infty}$ & 1.765   & 1.805   & 1.657   & 1.763   &  & 1.109   & 1.134   & 1.069   & 1.108   &  & 0.642   & 0.653   & 0.631   & 0.643   \\
    &                     & (0.028) & (0.028) & (0.027) & (0.028) &  & (0.017) & (0.017) & (0.017) & (0.017) &  & (0.010) & (0.010) & (0.010) & (0.010) \\
    & $L_{2,2}$           & 1.407   & 1.443   & 1.321   & 1.406   &  & 0.887   & 0.909   & 0.856   & 0.886   &  & 0.504   & 0.514   & 0.495   & 0.504   \\
    &                     & (0.022) & (0.022) & (0.021) & (0.022) &  & (0.014) & (0.013) & (0.013) & (0.014) &  & (0.007) & (0.007) & (0.007) & (0.007) \\
500 & $L_2$               & 7.527   & 7.783   & 7.209   & 7.505   &  & 5.079   & 5.177   & 4.975   & 5.074   &  & 3.133   & 3.160   & 3.107   & 3.131   \\
    &                     & (0.029) & (0.030) & (0.026) & (0.029) &  & (0.015) & (0.016) & (0.015) & (0.015) &  & (0.010) & (0.010) & (0.009) & (0.010) \\
    & $L_{\infty}$        & 1.030   & 1.066   & 0.974   & 1.028   &  & 0.626   & 0.646   & 0.606   & 0.625   &  & 0.359   & 0.367   & 0.354   & 0.359   \\
    &                     & (0.017) & (0.017) & (0.016) & (0.017) &  & (0.010) & (0.010) & (0.010) & (0.010) &  & (0.006) & (0.006) & (0.006) & (0.006) \\[-9pt]
    &                     &         &         &         &         &  &         &         &         &         &  &         &         &         &         \\
\hline
\end{tabular}
\end{sidewaystable}

\begin{sidewaystable}
\def\arraystretch{0.98}
\tabcolsep=5.5pt
\caption{Simulation results for AR(4) model based on 100 replications;
figures in parentheses indicate standard errors}\label{table2}
\begin{tabular}{cccccccccccccccc}
\hline
& & \multicolumn{4}{c}{$n = 100$} & & \multicolumn{4}{c}{$n = 200$} &
& \multicolumn{4}{c}{$n = 500$}\\
\cline{3-6} \cline{8-11} \cline{13-16}
$p$ & Norm & MLE & $\bOmega_{L_2}^{\mathrm{B}}$ & $\bOmega
_{L_1}^{\mathrm{B}}$ & Cholesky & & MLE & $\bOmega_{L_2}^{\mathrm{B}}$
& $\bOmega_{L_1}^{\mathrm{B}}$ & Cholesky & & MLE & $\bOmega
_{L_2}^{\mathrm{B}}$ & $\bOmega_{L_1}^{\mathrm{B}}$ & Cholesky \\
\hline
    &                     &         &         &         &         &  &         &         &         &         &  &         &         &         &         \\[-9pt]
    & $L_{\infty,\infty}$ & 1.836   & 2.066   & 1.758   & 1.821   &  & 1.078   & 1.177   & 1.053   & 1.076   &  & 0.642   & 0.673   & 0.636   & 0.641   \\
    &                     & (0.040) & (0.041) & (0.038) & (0.038) &  & (0.020) & (0.021) & (0.019) & (0.020) &  & (0.011) & (0.011) & (0.011) & (0.011) \\
    & $L_{2,2}$           & 1.158   & 1.340   & 1.101   & 1.149   &  & 0.672   & 0.754   & 0.654   & 0.672   &  & 0.399   & 0.426   & 0.394   & 0.399   \\
    &                     & (0.027) & (0.028) & (0.025) & (0.025) &  & (0.014) & (0.015) & (0.014) & (0.014) &  & (0.008) & (0.009) & (0.008) & (0.008) \\
50  & $L_2$               & 2.539   & 2.951   & 2.463   & 2.526   &  & 1.635   & 1.789   & 1.612   & 1.631   &  & 0.988   & 1.030   & 0.983   & 0.987   \\
    &                     & (0.027) & (0.030) & (0.025) & (0.026) &  & (0.015) & (0.018) & (0.014) & (0.015) &  & (0.008) & (0.009) & (0.008) & (0.008) \\
    & $L_{\infty}$        & 0.574   & 0.692   & 0.554   & 0.572   &  & 0.326   & 0.378   & 0.320   & 0.325   &  & 0.180   & 0.196   & 0.179   & 0.180   \\
    &                     & (0.014) & (0.015) & (0.014) & (0.014) &  & (0.007) & (0.007) & (0.007) & (0.007) &  & (0.003) & (0.004) & (0.003) & (0.004) \\[4pt]
    & $L_{\infty,\infty}$ & 1.993   & 2.231   & 1.907   & 1.973   &  & 1.210   & 1.315   & 1.182   & 1.209   &  & 0.702   & 0.738   & 0.694   & 0.702   \\
    &                     & (0.037) & (0.038) & (0.035) & (0.036) &  & (0.018) & (0.019) & (0.017) & (0.018) &  & (0.010) & (0.010) & (0.010) & (0.010) \\
    & $L_{2,2}$           & 1.263   & 1.451   & 1.200   & 1.255   &  & 0.761   & 0.849   & 0.738   & 0.760   &  & 0.440   & 0.471   & 0.434   & 0.440   \\
    &                     & (0.025) & (0.025) & (0.023) & (0.024) &  & (0.012) & (0.013) & (0.012) & (0.012) &  & (0.008) & (0.008) & (0.007) & (0.008) \\
100 & $L_2$               & 3.626   & 4.220   & 3.518   & 3.607   &  & 2.337   & 2.561   & 2.303   & 2.331   &  & 1.408   & 1.466   & 1.400   & 1.407   \\
    &                     & (0.028) & (0.032) & (0.026) & (0.028) &  & (0.014) & (0.016) & (0.013) & (0.014) &  & (0.008) & (0.009) & (0.008) & (0.008) \\
    & $L_{\infty}$        & 0.626   & 0.749   & 0.605   & 0.625   &  & 0.357   & 0.411   & 0.351   & 0.356   &  & 0.196   & 0.215   & 0.194   & 0.196   \\
    &                     & (0.015) & (0.015) & (0.014) & (0.014) &  & (0.006) & (0.006) & (0.006) & (0.006) &  & (0.003) & (0.003) & (0.003) & (0.003) \\[4pt]
    & $L_{\infty,\infty}$ & 2.165   & 2.413   & 2.069   & 2.145   &  & 1.324   & 1.435   & 1.292   & 1.319   &  & 0.763   & 0.802   & 0.754   & 0.762   \\
    &                     & (0.034) & (0.035) & (0.032) & (0.033) &  & (0.018) & (0.019) & (0.018) & (0.018) &  & (0.011) & (0.011) & (0.011) & (0.011) \\
    & $L_{2,2}$           & 1.376   & 1.569   & 1.307   & 1.363   &  & 0.841   & 0.932   & 0.816   & 0.838   &  & 0.479   & 0.512   & 0.472   & 0.478   \\
    &                     & (0.022) & (0.022) & (0.021) & (0.021) &  & (0.013) & (0.013) & (0.012) & (0.013) &  & (0.008) & (0.008) & (0.008) & (0.008) \\
200 & $L_2$               & 5.145   & 5.988   & 4.992   & 5.116   &  & 3.332   & 3.652   & 3.283   & 3.324   &  & 1.995   & 2.079   & 1.984   & 1.994   \\
    &                     & (0.028) & (0.032) & (0.026) & (0.028) &  & (0.015) & (0.017) & (0.014) & (0.015) &  & (0.007) & (0.008) & (0.007) & (0.007) \\
    & $L_{\infty}$        & 0.695   & 0.821   & 0.671   & 0.689   &  & 0.393   & 0.449   & 0.386   & 0.393   &  & 0.215   & 0.235   & 0.213   & 0.215   \\
    &                     & (0.013) & (0.014) & (0.013) & (0.013) &  & (0.006) & (0.006) & (0.006) & (0.006) &  & (0.003) & (0.004) & (0.003) & (0.003) \\[4pt]
    & $L_{\infty,\infty}$ & 2.476   & 2.732   & 2.362   & 2.447   &  & 1.482   & 1.599   & 1.444   & 1.480   &  & 0.833   & 0.875   & 0.823   & 0.830   \\
    &                     & (0.035) & (0.036) & (0.034) & (0.034) &  & (0.018) & (0.018) & (0.018) & (0.019) &  & (0.010) & (0.011) & (0.010) & (0.010) \\
    & $L_{2,2}$           & 1.579   & 1.778   & 1.498   & 1.562   &  & 0.946   & 1.039   & 0.918   & 0.945   &  & 0.526   & 0.561   & 0.518   & 0.524   \\
    &                     & (0.023) & (0.024) & (0.022) & (0.023) &  & (0.014) & (0.014) & (0.014) & (0.014) &  & (0.007) & (0.007) & (0.006) & (0.006) \\
500 & $L_2$               & 8.205   & 9.553   & 7.957   & 8.159   &  & 5.287   & 5.793   & 5.210   & 5.273   &  & 3.161   & 3.296   & 3.143   & 3.158   \\
    &                     & (0.026) & (0.030) & (0.024) & (0.026) &  & (0.016) & (0.018) & (0.015) & (0.016) &  & (0.007) & (0.008) & (0.006) & (0.007) \\
    & $L_{\infty}$        & 0.787   & 0.916   & 0.759   & 0.779   &  & 0.434   & 0.491   & 0.426   & 0.434   &  & 0.244   & 0.265   & 0.242   & 0.244   \\
    &                     & (0.013) & (0.013) & (0.012) & (0.013) &  & (0.006) & (0.006) & (0.006) & (0.006) &  & (0.004) & (0.004) & (0.004) & (0.004) \\[-9pt]
    &                     &         &         &         &         &  &         &         &         &         &  &         &         &         &         \\
\hline
\end{tabular}
\end{sidewaystable}

\begin{sidewaystable}
\def\arraystretch{0.98}
\tabcolsep=5.5pt
\caption{Simulation results for Fractional Gaussian noise model based
on 100 replications; figures in parentheses indicate standard errors}\label{table3}
\begin{tabular}{cccccccccccccccc}
\hline
& & \multicolumn{4}{c}{$n = 100$} & & \multicolumn{4}{c}{$n = 200$} &
& \multicolumn{4}{c}{$n = 500$}\\
\cline{3-6} \cline{8-11} \cline{13-16}
$p$ & Norm & MLE & $\bOmega_{L_2}^{\mathrm{B}}$ & $\bOmega
_{L_1}^{\mathrm{B}}$ & Cholesky & & MLE & $\bOmega_{L_2}^{\mathrm{B}}$
& $\bOmega_{L_1}^{\mathrm{B}}$ & Cholesky & & MLE & $\bOmega
_{L_2}^{\mathrm{B}}$ & $\bOmega_{L_1}^{\mathrm{B}}$ & Cholesky \\
\hline
    &                     &         &         &         &         &  &         &         &         &         &  &         &         &         &         \\[-9pt]
    & $L_{\infty,\infty}$ & 1.530   & 1.588   & 1.493   & 1.527   &  & 1.184   & 1.213   & 1.170   & 1.182   &  & 0.969   & 0.981   & 0.965   & 0.969   \\
    &                     & (0.025) & (0.025) & (0.024) & (0.024) &  & (0.015) & (0.015) & (0.014) & (0.015) &  & (0.008) & (0.008) & (0.008) & (0.008) \\
    & $L_{2,2}$           & 0.849   & 0.902   & 0.820   & 0.846   &  & 0.587   & 0.614   & 0.577   & 0.586   &  & 0.412   & 0.422   & 0.409   & 0.412   \\
    &                     & (0.019) & (0.019) & (0.018) & (0.019) &  & (0.012) & (0.012) & (0.011) & (0.011) &  & (0.005) & (0.005) & (0.005) & (0.005) \\
50  & $L_2$               & 2.169   & 2.271   & 2.116   & 2.164   &  & 1.666   & 1.706   & 1.646   & 1.665   &  & 1.329   & 1.340   & 1.323   & 1.329   \\
    &                     & (0.020) & (0.022) & (0.020) & (0.020) &  & (0.012) & (0.013) & (0.012) & (0.012) &  & (0.006) & (0.006) & (0.006) & (0.006) \\
    & $L_{\infty}$        & 0.570   & 0.615   & 0.552   & 0.569   &  & 0.360   & 0.376   & 0.355   & 0.359   &  & 0.221   & 0.226   & 0.219   & 0.221   \\
    &                     & (0.015) & (0.015) & (0.014) & (0.015) &  & (0.008) & (0.008) & (0.008) & (0.008) &  & (0.003) & (0.004) & (0.003) & (0.003) \\[4pt]
    & $L_{\infty,\infty}$ & 1.687   & 1.745   & 1.647   & 1.682   &  & 1.316   & 1.345   & 1.301   & 1.315   &  & 1.058   & 1.069   & 1.053   & 1.058   \\
    &                     & (0.024) & (0.024) & (0.023) & (0.024) &  & (0.014) & (0.014) & (0.014) & (0.014) &  & (0.007) & (0.007) & (0.007) & (0.007) \\
    & $L_{2,2}$           & 0.939   & 0.992   & 0.907   & 0.937   &  & 0.645   & 0.672   & 0.633   & 0.645   &  & 0.441   & 0.451   & 0.438   & 0.441   \\
    &                     & (0.018) & (0.018) & (0.017) & (0.018) &  & (0.011) & (0.011) & (0.011) & (0.011) &  & (0.005) & (0.005) & (0.005) & (0.005) \\
100 & $L_2$               & 3.076   & 3.222   & 3.000   & 3.068   &  & 2.351   & 2.406   & 2.322   & 2.351   &  & 1.886   & 1.902   & 1.876   & 1.885   \\
    &                     & (0.021) & (0.023) & (0.020) & (0.022) &  & (0.013) & (0.014) & (0.012) & (0.013) &  & (0.006) & (0.007) & (0.006) & (0.006) \\
    & $L_{\infty}$        & 0.634   & 0.680   & 0.614   & 0.632   &  & 0.395   & 0.413   & 0.388   & 0.394   &  & 0.238   & 0.243   & 0.236   & 0.238   \\
    &                     & (0.015) & (0.016) & (0.015) & (0.015) &  & (0.008) & (0.008) & (0.007) & (0.008) &  & (0.003) & (0.003) & (0.003) & (0.003) \\[4pt]
    & $L_{\infty,\infty}$ & 1.855   & 1.914   & 1.811   & 1.857   &  & 1.446   & 1.475   & 1.430   & 1.446   &  & 1.134   & 1.145   & 1.129   & 1.134   \\
    &                     & (0.022) & (0.022) & (0.021) & (0.022) &  & (0.014) & (0.014) & (0.014) & (0.014) &  & (0.008) & (0.008) & (0.008) & (0.008) \\
    & $L_{2,2}$           & 1.024   & 1.078   & 0.989   & 1.025   &  & 0.705   & 0.733   & 0.693   & 0.705   &  & 0.471   & 0.481   & 0.468   & 0.471   \\
    &                     & (0.016) & (0.016) & (0.016) & (0.016) &  & (0.011) & (0.011) & (0.011) & (0.011) &  & (0.005) & (0.005) & (0.005) & (0.005) \\
200 & $L_2$               & 4.348   & 4.555   & 4.239   & 4.340   &  & 3.335   & 3.414   & 3.294   & 3.334   &  & 2.659   & 2.680   & 2.645   & 2.659   \\
    &                     & (0.020) & (0.021) & (0.019) & (0.020) &  & (0.013) & (0.014) & (0.013) & (0.013) &  & (0.006) & (0.006) & (0.006) & (0.006) \\
    & $L_{\infty}$        & 0.712   & 0.760   & 0.689   & 0.711   &  & 0.439   & 0.460   & 0.431   & 0.439   &  & 0.257   & 0.262   & 0.255   & 0.256   \\
    &                     & (0.015) & (0.015) & (0.014) & (0.015) &  & (0.008) & (0.008) & (0.008) & (0.008) &  & (0.003) & (0.004) & (0.003) & (0.004) \\[4pt]
    & $L_{\infty,\infty}$ & 2.059   & 2.116   & 2.008   & 2.056   &  & 1.565   & 1.595   & 1.548   & 1.566   &  & 1.219   & 1.231   & 1.214   & 1.219   \\
    &                     & (0.022) & (0.022) & (0.022) & (0.022) &  & (0.012) & (0.012) & (0.012) & (0.013) &  & (0.007) & (0.007) & (0.007) & (0.007) \\
    & $L_{2,2}$           & 1.156   & 1.208   & 1.116   & 1.151   &  & 0.773   & 0.800   & 0.759   & 0.773   &  & 0.507   & 0.517   & 0.504   & 0.507   \\
    &                     & (0.017) & (0.016) & (0.016) & (0.017) &  & (0.010) & (0.010) & (0.010) & (0.010) &  & (0.004) & (0.005) & (0.004) & (0.005) \\
500 & $L_2$               & 6.865   & 7.190   & 6.694   & 6.850   &  & 5.266   & 5.386   & 5.201   & 5.263   &  & 4.215   & 4.249   & 4.194   & 4.215   \\
    &                     & (0.022) & (0.023) & (0.021) & (0.022) &  & (0.012) & (0.013) & (0.012) & (0.012) &  & (0.007) & (0.007) & (0.007) & (0.007) \\
    & $L_{\infty}$        & 0.802   & 0.851   & 0.776   & 0.801   &  & 0.482   & 0.505   & 0.474   & 0.482   &  & 0.282   & 0.288   & 0.280   & 0.282   \\
    &                     & (0.014) & (0.014) & (0.014) & (0.014) &  & (0.007) & (0.007) & (0.007) & (0.007) &  & (0.004) & (0.004) & (0.004) & (0.004) \\[-9pt]
    &                     &         &         &         &         &  &         &         &         &         &  &         &         &         &         \\
\hline
\end{tabular}
\end{sidewaystable}

Data is simulated from $\mathrm{N}_p(0,\bSigma)$, assuming specific
structures of the covariance $\bSigma$ or the precision $\bOmega$. For
all simulations, we compute the $L_{\infty}$-operator norm,
$L_2$-operator norm, $L_2$-norm and $L_{\infty}$-norm of the difference
between the estimate and the true parameter for sample sizes $n= 100,
200, 500$ and $p= 50,100,200,500$, representing cases like $p<n$, $p
\sim n$, $p > n$ and $p \gg n$. We simulate 100 replications in each
cases. Some of the simulation models are the same as those in \cite
{bickel2008regularized}.

\begin{example}[Autoregressive process: AR(1) covariance structure] \rm
Let the true covariance matrix have entries given by
%
\begin{equation}
\sigma_{ij} = \rho^{|i-j|},\,1\leq i,j\leq p,
\end{equation}
with $\rho=0.3$ in our simulation experiment.
The precision matrix is banded in this case, with banding parameter $1$.
\end{example}

\begin{example}[Autoregressive process: AR(4) covariance structure]\rm
The elements of true precision matrix are given by
%
\begin{equation}
\begin{split}
\omega_{ij} ={}& \Ind{(|i-j|=0)} + 0.4\,\Ind{(|i-j|=1)} + 0.2\,\Ind
{(|i-j|=2)} \\
&\qquad+ 0.2\,\Ind{(|i-j|=3)} + 0.1\,\Ind{(|i-j|=4)}.
\end{split}
\end{equation}
This precision matrix corresponds to an AR(4) process.
\end{example}

\begin{example}[Long range dependence]
We consider a fractional Gaussian Noise process, that is, the increment
process of fractional Brownian motion. The elements of the true
covariance matrix are given by
%
\begin{equation}
\sigma_{ij} = \frac{1}{2}[\|i-j|+1|^{2H} - 2|i-j|^{2H} + \|
i-j|-1|^{2H}], \, 1\leq i,j\leq p,
\end{equation}
where $H \in[0.5,1]$ is the Hurst parameter.
We take $H=0.7$ in the simulation example. This precision matrix does
not fall in the polynomial smoothness class used in the theorems. We
include this example in the simulation study to check how the proposed
method is performing when the assumptions of the theorems are not met.
\end{example}

Tables \ref{table1}--\ref{table3} show the simulation results for the different
scenarios and compare the performance of the Bayes estimators with the
graphical MLE and the banded estimator obtained in \cite
{bickel2008regularized} based on a modified Cholesky decomposition. The
banding parameter $k$ is chosen using the bandwidth estimation method
discussed in Section \ref{sec:bandestim}. Figure \ref{fig:ar4} shows the
log-posterior probabilities of the graphs corresponding to banding
parameter $k$ for prior distribution of $k$ given by $\rho_k \propto
\exp(-k^4)$.

\begin{figure}[t!]
\centering
\subfigure[AR(4) model, $n=100,p=50$]{\includegraphics[width=0.48\textwidth]{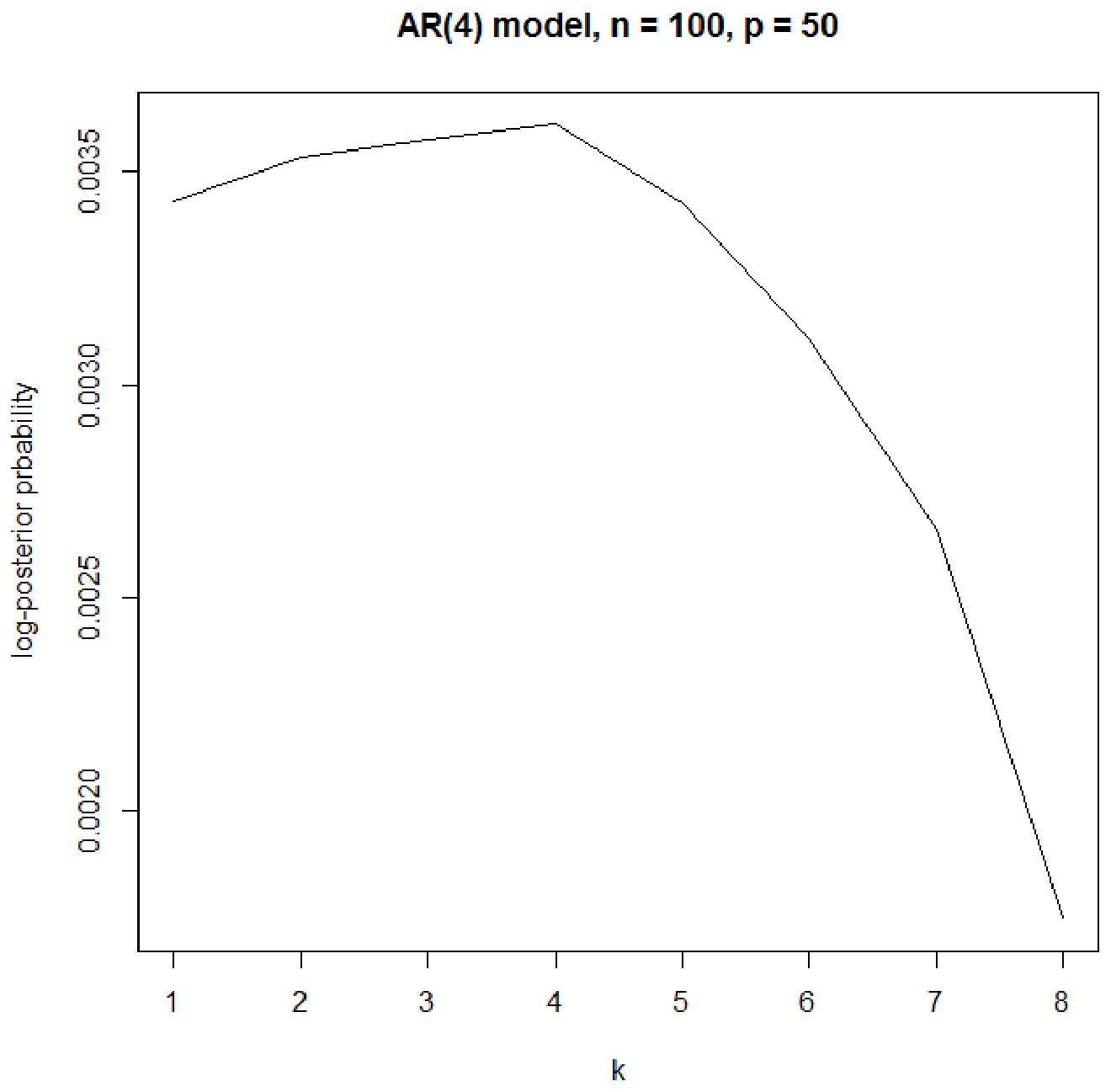}}\quad
\subfigure[AR(4) model, $n=100,p=100$]{\includegraphics[width=0.48\textwidth]{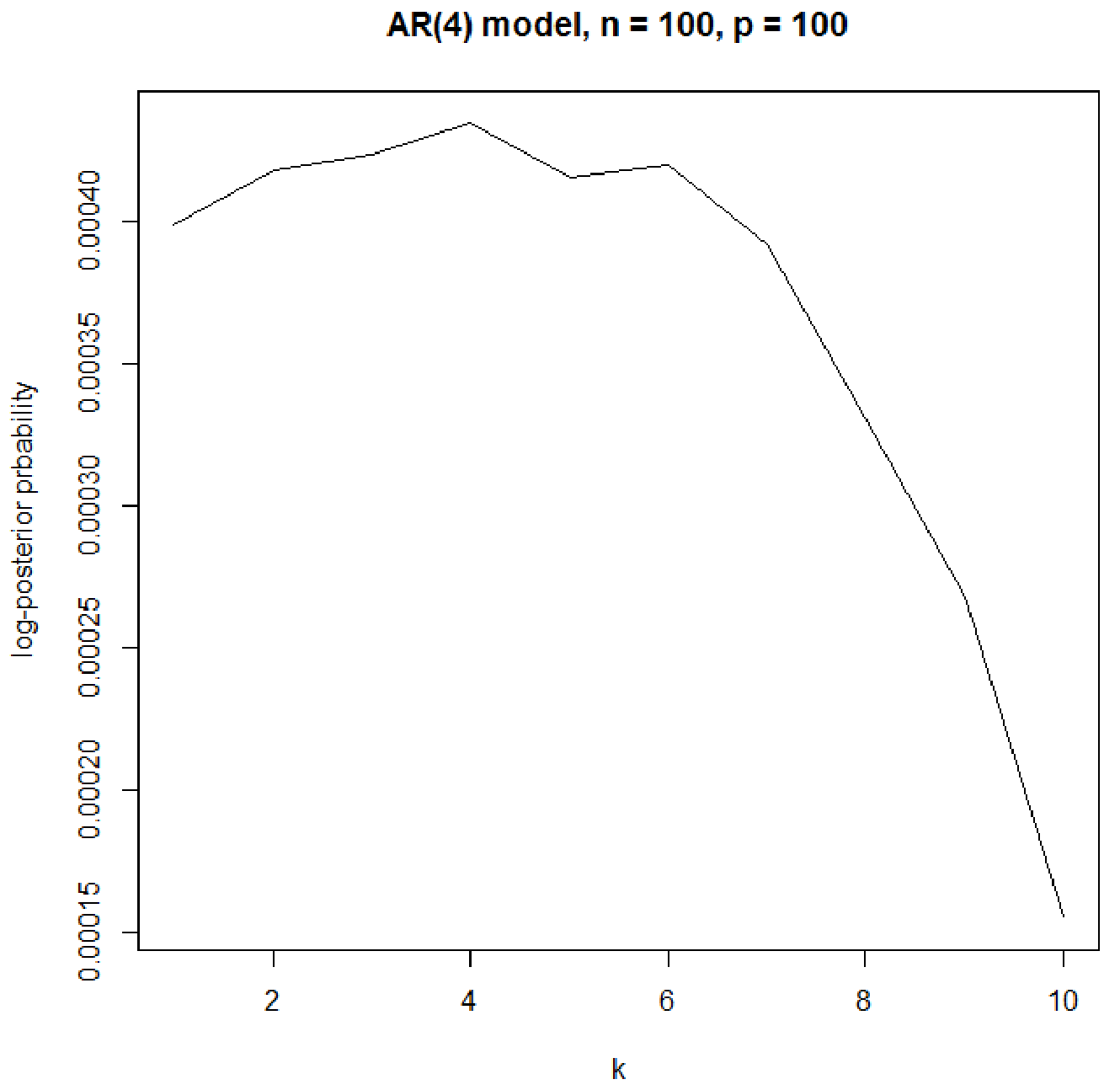}}\\
\subfigure[AR(4) model, $n=100,p=200$]{\includegraphics[width=0.48\textwidth]{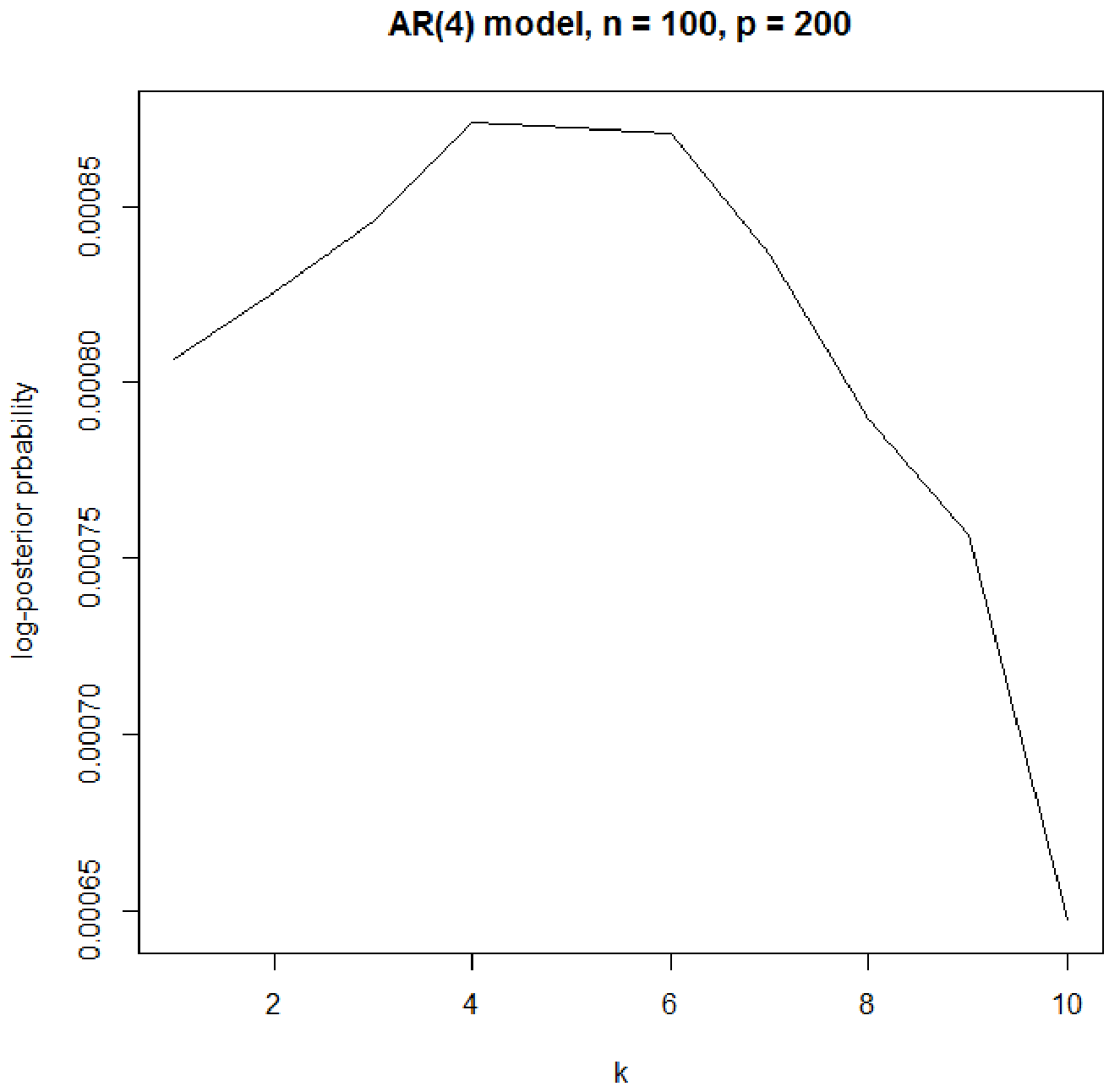}}\quad
\subfigure[AR(4) model, $n=100,p=500$]{\includegraphics[width=0.48\textwidth]{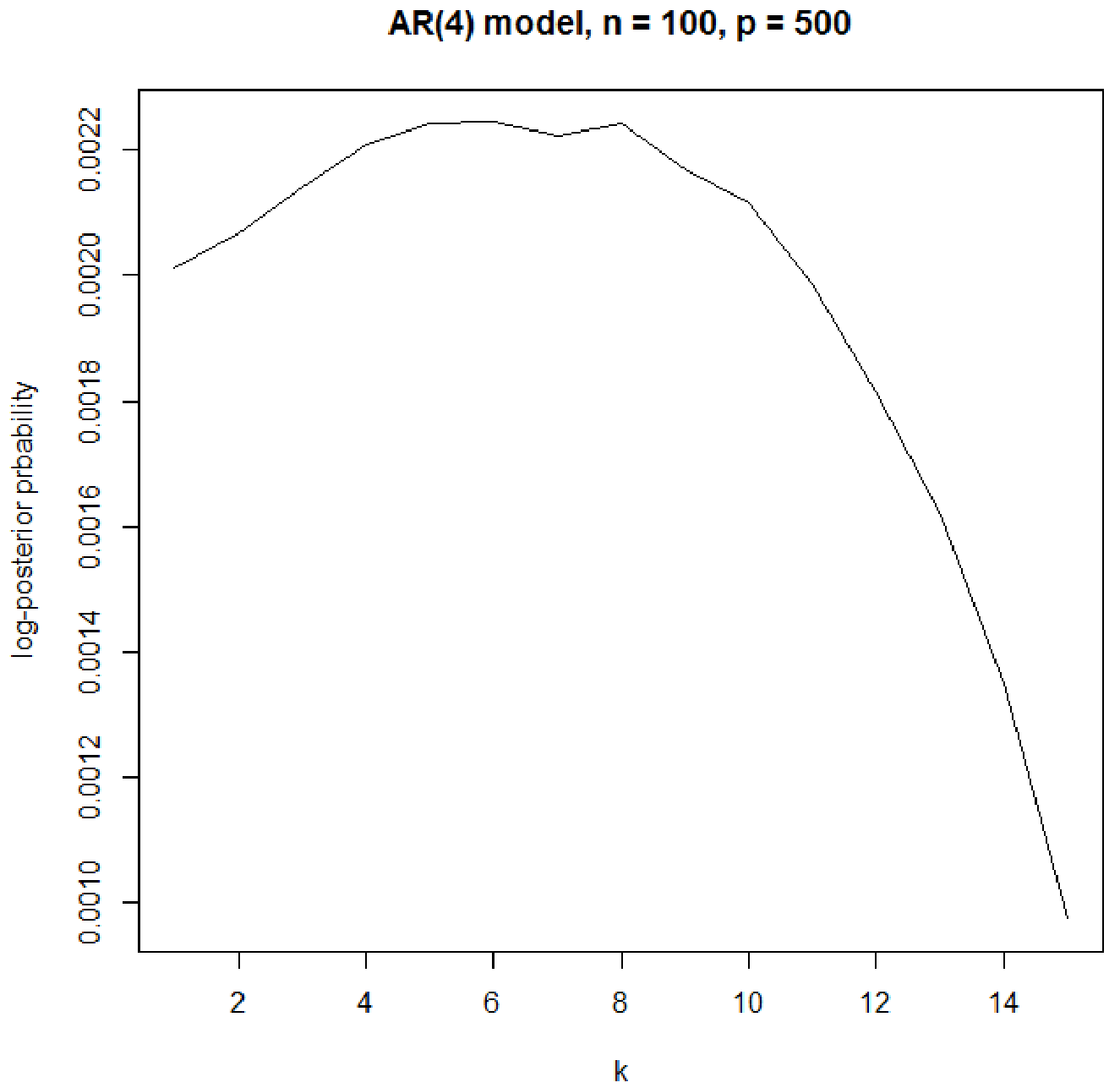}}
\vspace*{-3pt}
\caption{Figures showing log-posterior probabilities of graphs
corresponding to different banding parameters $k$. The graphs are
trimmed for larger values of $k$ as the log-posterior probabilities
decay further.}
\label{fig:ar450}\label{fig:ar4100}\label{fig:ar4200}\label{fig:ar4500}
\label{fig:ar4}\vspace*{-3pt}
\end{figure}

\section{Proofs}
\label{sec:proofs}

In this section we provide the proofs of the theorems and lemmas stated
in Section \ref{sec:mainresults}. Proofs of these results will require some additional
lemmas and propositions, which we include in the
\hyperref[app]{Appendix}.\vadjust{\goodbreak}

\begin{proof}[Proof of Proposition \ref{theorem:mleconvergence}]
Let the true covariance matrix be denoted by $\bSigma_0$, that is,
$\bOmega_0 = \bSigma_0^{-1}$ and $\bGamma_0$ is the inverse of the
banded version of $\bOmega_0$, that is $\bGamma_0 = (B_k(\bOmega
_0))^{-1}$. This is well defined for sufficiently large $k$ since
$B_k(\bOmega_0)$ is close to the non-singular matrix $\bOmega_0$ (even
in $L_\infty$ operator norm).
The $L_{\infty}$-operator norm of the difference between the graphical
MLE $\widehat{\bOmega}^{\mathrm{M}}$ and the true precision matrix
$\bOmega_0$ can be written as
%
\begin{equation}
\label{eqn:bound1}
\|\widehat{\bOmega}^{\mathrm{M}} - \bOmega_0\|_{(\infty,\infty)}
\leq\|
\widehat{\bOmega}^{\mathrm{M}} - B_k(\bOmega_0)\|_{(\infty,\infty
)} + \|
\bOmega_0 - B_k(\bOmega_0)\|_{(\infty,\infty)}.
\end{equation}
As shown in \cite{lauritzen1996graphical}, in a decomposable Gaussian
graphical model with precision matrix $\bOmega= \bSigma^{-1}$, we have,
\[
\sum_{j=1}^{p-k}\left(\bOmega_{C_j}\right)^0 - \sum
_{j=2}^{p-k}\left
(\bOmega_{S_j}\right)^0=
\sum_{j=1}^{p-k}\left(\bSigma_{C_j}^{-1}\right)^0 - \sum
_{j=2}^{p-k}\left(\bSigma_{S_j}^{-1}\right)^0.
\]
In our case we shall be working with $B_k(\bOmega_0)$ and the
corresponding inverse matrix $\bGamma_0$. Using the above
representation and also the form of the graphical MLE, the first term
on the right side of the bound in equation (\ref{eqn:bound1}) can be
written as,
\begin{eqnarray*}
\lefteqn{ \Biggl\|\sum_{j=1}^{p-k}\bigl(\bmS_{C_j}^{-1}\bigr)^0 -
\sum
_{j=2}^{p-k}\bigl(\bmS_{S_j}^{-1}\bigr)^0 - \sum_{j=1}^{p-k}
\bigl(\bOmega_{0,C_j}\bigr)^0 + \sum_{j=2}^{p-k}\bigl(\bOmega
_{0,S_j}\bigr)^0\Biggr\|_{(\infty,\infty)} }\nn\\
=\lefteqn{ \Biggl\|\sum_{j=1}^{p-k}\left(\bmS_{C_j}^{-1}\right)^0
- \sum
_{j=2}^{p-k}\bigl(\bmS_{S_j}^{-1}\bigr)^0 - \sum_{j=1}^{p-k}
\bigl(\bGamma_{0,C_j}^{-1}\bigr)^0 + \sum_{j=2}^{p-k}\bigl(\bGamma
_{0,S_j}^{-1}\bigr)^0\Biggr\|_{(\infty,\infty)} }\nn\\
&&\le\Biggl\|\sum_{j=1}^{p-k}\left(\!\bigl(\bmS_{C_j}^{-1}\bigr)^0 -
\bigl(\bGamma_{0,C_j}^{-1}\bigr)^0\right)\Biggr\|_{(\infty
,\infty)}
+ \Biggl\|\sum_{j=2}^{p-k}\left(\!\bigl(\bmS_{S_j}^{-1}\bigr)^0 -
\bigl(\bGamma_{0,S_j}^{-1}\bigr)^0\right)\Biggr\|_{(\infty,\infty)}.
\end{eqnarray*}
Using the fact that there are only $(2k+1)$ terms in above expressions
inside the norms which have a given row non-zero, it follows that
%
\begin{eqnarray}
\label{eqn:rhsterm1full}
\lefteqn{
\left\|\sum_{j=1}^{p-k}\left\{\left(\bmS_{C_j}^{-1}\right)^0 -
\left
(\bGamma_{0,C_j}^{-1}\right)^0\right\}\right\|_{(\infty,\infty)}
}\nn\\
&&=
\mathop{\max}_{l}\sum_{l'}\left|\left[\sum_{j=1}^{p-k}\left\{
\left(\bmS
_{C_j}^{-1}\right)^0 - \left(\bGamma_{0,C_j}^{-1}\right)^0\right\}
\right
]_{(l,l')}\right| \nonumber\\
&&\leq\mathop{\max}_{l}\sum_{j=1}^{p-k}\sum_{l'}\left|\left
[\left
(\bmS_{C_j}^{-1}\right)^0 - \left(\bGamma_{0,C_j}^{-1}\right
)^0\right
]_{(l,l')}\right| \nonumber\\
&& \leq(2k+1)\mathop{\max}_{j}\mathop{\max}_{l}\sum_{l'}\,\left
|\left[\bmS_{C_j}^{-1} - \bGamma_{0,C_j}^{-1}\right]_{(l,l')}\right|
\nonumber\\
&& = (2k+1)\mathop{\max}_{j}\left\|\bmS_{C_j}^{-1} - \bGamma
_{0,C_j}^{-1}\right\|_{(\infty,\infty)} \nn\\
&& \lesssim k^{3/2}\mathop{\max}_{j}\left\|\bmS_{C_j}^{-1} -
\bGamma
_{0,C_j}^{-1}\right\|_{(2,2)},
\label{new1}
\end{eqnarray}
where the subscript $(l,l')$ on the matrices above stand for their
respective $(l,l')$th entries.
Using the multiplicative inequality $\|\bmA\bmB\|\le\|\bmA\|\|\bmB
\|$
of operator norms, we have
%
\begin{eqnarray}
\label{eqn:rhsterm1}
\lefteqn{\mathop{\max}_{j}\left\|\bmS_{C_j}^{-1} - \bGamma
_{0,C_j}^{-1}\right\|_{(2,2)} }\nn\\
&&= \mathop{\max}_{j}\left\|\bGamma_{0,C_j}^{-1}(\bGamma_{0,C_j} -
\bmS
_{C_j})\bmS_{C_j}^{-1}\right\|_{(2,2)} \nn\\
&& \leq\mathop{\max}_{j}\left\{\left\|\bGamma_{0,C_j}^{-1}\right
\|
_{(2,2)}\left\|\bGamma_{0,C_j} - \bmS_{C_j}\right\|_{(2,2)}\left\|
\bmS
_{C_j}^{-1}\right\|_{(2,2)}\right\}.
\end{eqnarray}
The middle term $\|\bGamma_{0,C_j} - \bmS_{C_j}\|_{(2,2)}$
in the above expression is bounded by
%
\begin{equation}
\label{eqn:midterm}
\left\|\bSigma_{0,C_j} - \bmS_{C_j}\right\|_{(2,2)} + \left\|
\bGamma
_{0,C_j} - \bSigma_{0,C_j}\right\|_{(2,2)}.
\end{equation}

We now estimate the norm difference between the matrices $\bGamma_0$
and $\bSigma_0$. We have,
%
\begin{eqnarray}
\label{eqn:GammaSigmanormbound}
{ \left\| \bGamma_0 - \bSigma_0 \right\|_{(2,2)} }
& = &\left\| \bGamma_0 (\bSigma_0^{-1} - \bGamma_0^{-1}) \bSigma_0
\right\|_{(2,2)} \nn\\
& = &\left\| \bGamma_0 \left(\bOmega_0 - B_k(\bOmega_0) \right)
\bSigma
_0 \right\|_{(2,2)} \nn\\
& \leq&\left\|\bGamma_0\right\|_{(2,2)} \left\| \bOmega_0 -
B_k(\bOmega
_0) \right\|_{(2,2)} \left\| \bSigma_0 \right\|_{(2,2)}.
\end{eqnarray}

Now, $\| \bSigma_0\|_{(2,2)} = (\eig_1 (\bOmega_0))^{-1}
\leq\varepsilon_0^{-1}$ by assumption and
\[
\left\| \bOmega_0 - B_k(\bOmega_0) \right\|_{(2,2)} \leq\left\|
\bOmega
_0 - B_k(\bOmega_0) \right\|_{(\infty,\infty)} \leq\gamma(k)
\]
since $\bOmega_0 \in\mc{U}(\varepsilon_0,\gamma)$. Also,
$\|\bGamma_0\|_{(2,2)} = (\mathrm{eig}_1(B_k(\bOmega
_0)))^{-1}$. Note that $B_k(\bOmega_0) \geq\bOmega_0 - \gamma(k)\bm{I}$,
which has minimum eigenvalue bounded away from zero by the assumption
on $\bOmega_0$ and since $\gamma(k)\to0$. Thus, equation (\ref
{eqn:GammaSigmanormbound}) gives,
%
\begin{equation}
\label{eqn:GammaSigmanormbound2}
\left\| \bGamma_0 - \bSigma_0 \right\|_{(2,2)} = O(\gamma(k)).
\end{equation}
The above norm bound also applies to the difference of corresponding
submatrices defined by the cliques and separators of the graph.
Also note that
\[
\left\|\bGamma_{0,C_j}^{-1}\right\|_{(2,2)}\le\left\|\bOmega
_{0,C_j}\right\|_{(2,2)}\le\left\|\bOmega_0\right\|_{(2,2)}\le\ve_0^{-1}
\]
as $\bGamma_{0,C_j}^{-1}\le(B_k(\bOmega_0))_{C_j} = \bOmega
_{0,C_j}$ and
$\bOmega_0\in\mathcal{U}(\ve_0,\ga)$.
Further, $\|\bSigma_{0,C_j}\|_{(2,2)}\le\|\bSigma
_{0}\|_{(2,2)}
=(\eig_1 (\bOmega_0))^{-1}\le\ve_0^{-1}$. Thus applying
Lemma \ref{lemmaA4} with $r=2$ and $\bmD=\bSigma_{0,C_j}$, we obtain
\begin{eqnarray*}
\P\left(\mathop{\max}_{j}\left\|\bmS_{C_j}^{-1}\right\|_{(2,2)}
\geq
M_1\right)
&\leq& p\, \mathop{\max}_{j}\P\left(\left\|\bmS_{C_j}^{-1}\right
\|
_{(2,2)} \geq M_1\right)\\
&\leq& M_1' pk^2\exp[-m_1nk^{-2}]
\end{eqnarray*}
for some constant $M_1,M_1',m_1>0$, while from Lemma \ref{lemmaA3},
\[
\P\left(\mathop{\max}_{j}\left\|\bSigma_{0,C_j} - \bmS
_{C_j}\right\|
_{(2,2)} \geq t\right) \leq M_2 pk^2 \exp[-m_2 nk^{-2}t^2]
\]
for $|t| < m_2'$ for some constants $M_2,m_2,m_2'>0$.
Now choose $t = Ak(n^{-1}\log\,p)^{1/2}$ for some sufficiently large
$A$ to get the bound, using equations (\ref{eqn:rhsterm1full}), (\ref
{eqn:rhsterm1}), (\ref{eqn:midterm}) and (\ref{eqn:GammaSigmanormbound2}),
%
\begin{equation}
\label{eqn:mle1}
\left\|\sum_{j=1}^{p-k}\left(\left(\bmS_{C_j}^{-1}\right)^0 -
\left
(\bSigma_{0,C_j}^{-1}\right)^0\right)\right\|_{(\infty,\infty)}
=O_P\left(k^{5/2}(n^{-1}
\log p)^{1/2} + k^{3/2}\gamma(k)\right).
\end{equation}
By a similar argument, we can establish that
%
\begin{equation}
\label{eqn:mle1'}
\left\|\sum_{j=2}^{p-k}\left(\left(\bmS_{S_j}^{-1}\right)^0 -
\left
(\bSigma_{0,S_j}^{-1}\right)^0\right)\right\|_{(\infty,\infty
)}=O_P\left
(k^{5/2}(n^{-1}
\log p)^{1/2} + k^{3/2}\gamma(k)\right).
\end{equation}
Therefore, as $\|\bOmega_0 - B_k(\bOmega_0)\|_{(\infty,\infty)}
\leq
\gamma(k)$ for any $\bOmega_0\in\mathcal{U}(\ve_0,\ga)$, the
assertion follows.
\end{proof}

\begin{proof}[Proof of Lemma \ref{lemma:Bayesmle}]
We shall first prove the result for $\widehat{\bOmega}_{L_2}^{\mathrm
{B}}$. The $L_{\infty}$-operator norm of $\widehat{\bOmega
}_{L_2}^{\mathrm{B}} - \widehat{\bOmega}^{\mathrm{M}}$ can be
bounded by
%
\begin{eqnarray}
&& \rec{n}\left\|\sum_{j=2}^{p-k}\left((n^{-1}\bmI_{k}+\bmS
_{S_j})^{-1}\right)^0\right\|_{(\infty,\infty)} \label{term1} \\
&&{}+ \f{\delta+k+n}{n}\left\|\sum_{j=1}^{p-k}\left((n^{-1}\bmI
_{k+1}+\bmS
_{C_j})^{-1}\right)^0 - \sum_{j=1}^{p-k}(\bmS_{C_j}^{-1})^0 \right\|
_{(\infty,\infty)}\label{term2}\\
&&{}+ \f{\delta+k+n}{n}\left\|\sum_{j=2}^{p-k}\left((n^{-1}\bmI
_{k}+\bmS
_{S_j})^{-1}\right)^0 - \sum_{j=2}^{p-k}(\bmS_{S_j}^{-1})^0 \right\|
_{(\infty,\infty)}\label{term3} \\
&&{}+ \left|\f{\delta+k+n}{n}-1\right| \left\|\sum_{j=1}^{p-k}\left
(\bmS
_{C_j}^{-1}\right)^0 - \sum_{j=2}^{p-k}\left(\bmS_{S_j}^{-1}\right
)^0\right\|_{(\infty,\infty)}.\label{term4}
\end{eqnarray}
Now, the expression in (\ref{term1}) above is
\begin{eqnarray*}
\lefteqn{ \rec{n}\mathop{\max}_{l}\sum_{l'}\left|\left[\sum
_{j=2}^{p-k}\left((n^{-1}\bmI_{k}+\bmS_{S_j})^{-1}\right)^0\right
]_{(l,l')}\right|} \\
&&\leq\rec{n}\mathop{\max}_{l}\sum_{j=2}^{p-k}\sum_{l'}\left
|\left
[\left((n^{-1}\bmI_{k}+\bmS_{S_j})^{-1}\right)^0\right
]_{(l,l')}\right|
\\
&&\leq\f{2k+1}{n}\mathop{\max}_{j}\mathop{\max}_{l}\sum_{l'}\,
\left
|\left[(n^{-1}\bmI_{k}+\bmS_{S_j})^{-1}\right]_{(l,l')}\right| \\[-3pt]
&&= \f{2k+1}{n}\mathop{\max}_{j}\left\|(n^{-1}\bmI_{k}+\bmS
_{S_j})^{-1}\right\|_{(\infty,\infty)},
\end{eqnarray*}
which is bounded by a multiple of
%
\begin{equation}
\label{firstterm1}
\f{k^{3/2}}{n}\mathop{\max}_{j}\left\|(n^{-1}\bmI_{k}+\bmS
_{S_j})^{-1}\right\|_{(2,2)}
\leq\f{k^{3/2}}{n}\mathop{\max}_{j}\left\|\bmS_{S_j}^{-1}\right\|
_{(2,2)}.
\end{equation}
In view of Lemma~\ref{lemmaA4}, we have that for some $M_3,M_3',m_3>0$,
\[
\P\left(\mathop{\max}_{j}\left\|\bmS_{S_j}^{-1}\right\|_{(2,2)}
\geq
M_3\right)
\leq M_3' pk^2\exp[-m_3 nk^{-2}],
\]
which converges to zero if $k^2(\log p)/n \rightarrow0$.
This leads to the estimate
%
\begin{equation}
\label{firstterm}
n^{-1}\left|\left|\sum_{j=2}^{p-k}\left((n^{-1}\bmI_{k}
+\bmS_{S_j})^{-1}\right)^0\right|\right|_{(\infty,\infty)} =
O_P\left
(k^{3/2}/n\right).
\end{equation}
For (\ref{term2}), we observe that
\begin{eqnarray*}
\lefteqn{\left\|\sum_{j=1}^{p-k}\left((n^{-1}\bmI_{k+1}+\bmS
_{C_j})^{-1}\right)^0 - \sum_{j=1}^{p-k}(\bmS_{C_j}^{-1})^0 \right\|
_{(\infty,\infty)} }\nonumber\\[-3pt]
&&\leq(2k+1)\,\mathop{\max}_{j}\left\|(n^{-1}\bmI_{k+1}+\bmS
_{C_j})^{-1} - \bmS_{C_j}^{-1} \right\|_{(\infty,\infty)} \nonumber
\\[-3pt]
&&\lesssim k^{3/2}\,\mathop{\max}_{j}\,\left\|(n^{-1}\bmI
_{k+1}+\bmS
_{C_j})^{-1} - \bmS_{C_j}^{-1} \right\|_{(2,2)} \nonumber
\end{eqnarray*}
and that
\begin{eqnarray*}
\lefteqn{\left\|(n^{-1}\bmI_{k+1}+\bmS_{C_j})^{-1} - \bmS_{C_j}^{-1}
\right\|_{(2,2)} }\\[-3pt]
&&\le
\left\|(n^{-1}\bmI_{k+1}+\bmS_{C_j})^{-1}\right\|_{(2,2)}\left\|
n^{-1}\bmI_{k+1}\right\|_{(2,2)}
\left\|\bmS_{C_j}^{-1} \right\|_{(2,2)}\\[-3pt]
&&\leq n^{-1}\left\|\bmS_{C_j}^{-1} \right\|_{(2,2)}^2.
\end{eqnarray*}
Now under $k^2(\log p)/n \rightarrow0$, an application of Lemma~\ref
{lemmaA4} leads to the bound
$O_P(k^{3/2}/n)$ for (\ref{term2}).

A similar argument gives rise to the same $O_P(k^{3/2}/n)$ bound for
(\ref{term3}).

Finally to consider (\ref{term4}). As argued in bounding (\ref{term1}),
we have that
\begin{eqnarray*}
\lefteqn{\left\|\sum_{j=1}^{p-k}\left(\bmS_{C_j}^{-1}\right)^0 -
\sum
_{j=2}^{p-k}\left(\bmS_{S_j}^{-1}\right)^0\right\|_{(\infty,\infty
)}}\\[-3pt]
&&\leq k^{1/2}(2k+1)\left[\mathop{\max}_{j}\left\|\bmS
_{C_j}^{-1}\right\|_{(2,2)}
+ \mathop{\max}_{j} \left\|\bmS_{S_j}^{-1}\right\|_{(2,2)} \right]
=O_P(k^{3/2}),
\end{eqnarray*}
under the assumption $k^2(\log p)/n\to0$ by another application of
Lemma \ref{lemmaA4}.
Since $n^{-1}(\delta+k+n)-1=O(k/n)$, it follows that the expression in
(\ref{term4}) is $O_P(k^{5/2}/n)$, which is the weakest estimate among
all terms in the bound for $\|\widehat{\bOmega}_{L_2}^B-\widehat
{\bOmega
}^M\|_{(\infty,\infty)}$. The result thus follows for $\widehat
{\bOmega
}_{L_2}^{\mathrm{B}}$.

The assertion for the estimator $\widehat{\bOmega}_{L_1}^{\mathrm{B}}$
follows similarly.
\end{proof}

\begin{proof}[Proof of Proposition \ref{prop:Bayesconvergence}]
The proof follows from Theorem \ref{theorem:mleconvergence} and Lemma
\ref{lemma:Bayesmle} using the triangle inequality.
\end{proof}

\begin{proof}[Proof of Theorem \ref{theorem:Posteriorconvergence}]
The posterior distribution of the precision matrix $\bOmega$ given the
data $\bX$ is a $G$-Wishart distribution $W_G(\delta+n,\bmI_p +
n\bmS
)$. We can write $\bOmega$ as
%
\begin{equation}
\label{eqn:OmegaandW}
\bOmega= \sum_{j=1}^{p-k}\left(\bOmega_{C_j}\right)^0 - \sum
_{j=2}^{p-k}\left(\bOmega_{S_j}\right)^0 = \sum_{j=1}^{p-k}\left
(\bSigma
_{C_j}^{-1}\right)^0 - \sum_{j=2}^{p-k}\left(\bSigma
_{S_j}^{-1}\right)^0.
\end{equation}
The submatrix $\bSigma_{C_j}$ for any clique $C_j$ has a inverse
Wishart distribution with parameters $\delta+ n$ and scale matrix
$(\bmI_p + n\bmS)_{C_j}$, $j=1,\ldots,p-k$. Thus, $W_{C_j} = \bSigma
_{C_j}^{-1}$ has a Wishart distribution induced by the corresponding
inverse Wishart distribution. In particular, if $i\in C_j$, then $\tau
_{in}^{-1}w_{ii}$ has chi-square distribution with $(\de+n)$ degrees of
freedom, where $\tau_{in}$ is the $(i,i)$th entry of $((\bmI+\bmS
_{C_j})^{-1})^0$. Fix a clique $C=C_j$ and define $\bm{T}_{n}=\diag
(w_{ii}:\,i\in C)$. For $i,j\in C$, let $w_{ij}^*=w_{ij}/\sqrt{\tau
_{in}\tau_{jn}}$ and $\bm{W}_C^*=\defMatrix{w_{ij}^*:\, i,j\in C}$.
Then $\bm{W}_C^*$ given $\bX$ has a Wishart distribution with
parameters $\delta+ n$ and scale matrix $\bm{T}_{n}^{-1/2}(\bmI_{k+1}
+ n\bmS_{C})\bm{T}_{n}^{-1/2}$.

We first note that $\mathop{\max}_{i} \tau_{in} = O_P(n^{-1})$.
To see this, observe that $(\bmI_k+n\bm{S}_C)^{-1}\le n^{-1}\bm
{S}_C^{-1}$, so that
\[
\mathop{\max}_{i} |\tau_{in}| \leq\frac{1}{n}\|\bmS_C^{-1}\|
_{(2,2)}=O_P(n^{-1})
\]
in view of Lemma \ref{lemmaA4}.
On the other hand, from Lemma~\ref{lemmaA3}, it follows that $\max_C\|
\bmS_C\|_{(2,2)}=O_P(1)$, so with probability tending to one, $\bmS
_C\le L \bmI_C$, and hence $(\bmI+n\bmS)_C^{-1}\ge(1+n L)^{-1}\bmI_C$
simultaneously for all cliques, for some constant $L>0$. Hence $\max_i
\tau_{in}^{-1} = O_P(n)$.
Consequently, with probability tending to one, the maximum eigenvalue
of $\bm{T}_{n}^{-1/2}(\bmI_{k+1} + n\bmS_{C})\bm{T}_{n}^{-1/2}$ is
bounded by a constant depending only on $\ve_0$, simultaneously for all
cliques. Hence applying Lemma~A.3 of \cite{bickel2008regularized}, it
follows that
for all $i,j$,
%
\begin{equation}
\label{post2.2}
\P\left[|w_{ij} - \E(w_{ij}|\bX)|\geq t\right] \leq M_4\exp
[-m_4(\delta+n)t^2],
\quad|t|<m_4',
\end{equation}
for some constants $M_4,m_4,m_4'>0$ depending on $\ep_0$ only.

Now, as a $G$-Wishart prior gives rise to a $k$-banded structure, as
arguing in the bounding of (\ref{term1}) and using (\ref
{eqn:OmegaandW}), we have that, for some $M_5,m_5,m_5'>0$, and all $|t|<m_5'$,
%
\begin{equation}
\P\left\{\|\bOmega- \widehat{\bOmega}_{L_2}^{\mathrm{B}}\|
_{(\infty
,\infty)} \geq k^2t | \bX\right\} \leq M_5 p k^2 \exp[-m_5 n t^2].
\end{equation}
The reduction in the number of terms in the rows from $p$ to $(2k+1)$
is possible due to the fact that the $G$-Wishart posterior preserves
the banded structure of the precision matrix.
Choosing $t = A(n^{-1}\log p)^{1/2}$, with $A$ sufficiently large, we get
%
\begin{equation}
\label{eqn:post2.3}
\E_0[\P\{\|\bOmega- \widehat{\bOmega}_{L_2}^{\mathrm{B}}\|
_{(\infty
,\infty)} \geq A k^2(n^{-1}\log p)^{1/2} | \bX\}] \rightarrow0.
\end{equation}
Therefore, using Proposition \ref{prop:Bayesconvergence},
\begin{eqnarray*}
\lefteqn{
\E_0\left[\P\left\{\|\bOmega- \bOmega_0\|_{(\infty,\infty)} >
2\epsilon
_n | \bX\right\}\right] }\\
&&\leq\P_0\left\{\|\widehat{\bOmega}_{L_2}^{\mathrm{B}} - \bOmega
_0\|
_{(\infty,\infty)} > \epsilon_n | \bX\right\} + \E_0\left[\P
\left\{\|
\bOmega- \widehat{\bOmega}_{L_2}^{\mathrm{B}}\|_{(\infty,\infty)} >
\epsilon_n | \bX\right\}\right],
\end{eqnarray*}
which converges to zero if $\ep_n=A(k^{5/2}(n^{-1} \log p)^{1/2} +
k^{3/2}\gamma(k))$.
\end{proof}

\begin{proof}[Proof of (\ref{eqn:refprior})]
For the graph induced by banding, the posterior mean under the
reference prior is given by the expression
\[
\widehat{\bOmega}^R=
\E(\bOmega|\bmS) = \sum_{j=1}^{p-k}(\bmS_{C_j}^{-1})^0 -
(1-n^{-1})(\bmS
_{S_2}^{-1})^0 - (1-n^{-1})\sum_{j=3}^{p-k}(\bmS_{S_j}^{-1})^0.
\]
Therefore
\begin{eqnarray*}
\|\widehat{\bOmega}^{\mathrm{R}} - \widehat{\bOmega}^{\mathrm
{M}}\|
_{(\infty,\infty)} &= & \left\|n^{-1}(\bmS_{S_2}^{-1})^0 +
n^{-1}\sum
_{j=3}^{p-k}(\bmS_{S_j}^{-1})^0\right\|_{(\infty,\infty)} \\
&\leq& n^{-1}\left\|\sum_{j=2}^{p-k}(\bmS_{S_j}^{-1})^0\right\|
_{(\infty,\infty)} + n^{-1}\left\|(\bmS_{S_2}^{-1})^0\right\|
_{(\infty
,\infty)}.
\end{eqnarray*}
The rest of the proof proceeds as in Lemma \ref{lemma:Bayesmle}.
\end{proof}

\appendix

\section*{Appendix: Proofs of auxiliary results}\label{app}

In this section we give proofs of some lemmas we have used in the
paper, which are of some general interest.

The first lemma deals with the various equivalence conditions for
matrix norms and is easily found in standard textbooks.
\begin{lemma}
\label{lemmaA1}
For a symmetric matrix $\bmA$ of order $k$, we have the following:
\begin{enumerate}
\item$\|\bmA\|_{(2,2)} \leq\|\bmA\|_{(\infty,\infty)}\leq\sqrt
{k}\|
\bmA\|_{(2,2)}$;
\item$\|\bmA\|_\infty\leq\|\bmA\|_{(2,2)} \leq\|\bmA\|_{(\infty
,\infty)} \leq k\|\bmA\|_{\infty}$;
\end{enumerate}
\end{lemma}

Now we show that matrices $\bOmega$ belonging to the class $\mathcal
{U}(\varepsilon_0,\gamma)$ automatically have bounded $L_\iy
$-operator norm.

\begin{lemma}
\label{lemma:classequivalence}
For every $\varepsilon_0$, there exist $K$ depending only on $\ve_0$
and $\ga(\cdot)$ such that
for all $\Om\in\mathcal{U}(\varepsilon_0,\gamma)$, we have $\|
\bOmega\|
_{(\infty,\infty)}\leq K$.
\end{lemma}

\begin{proof}
For any fixed $k_0$,
%
\begin{eqnarray}
\|\bOmega\|_{(\infty,\infty)} &\leq& \|\bOmega- B_{k_0}(\bOmega)\|
_{(\infty,\infty)} + \|B_{k_0}(\bOmega)\|_{(\infty,\infty)}
\nonumber\\
&\leq& \gamma(k_0) + (2k_0 + 1)\|\bOmega\|_{\infty} \nonumber\\
&\leq& \gamma(k_0) + (2k_0 + 1)\|\bOmega\|_{(2,2)} \nonumber\\
&\leq& \gamma(k_0) + (2k_0 + 1)\varepsilon_0^{-1},
\end{eqnarray}
so the conclusion holds for $K = \gamma(k_0) + (2k_0 + 1)\varepsilon_0^{-1}$.
\end{proof}

\begin{lemma}
\label{lemmaA3}
Let $\bZ_i$, $i=1,\ldots,n$, be i.i.d. $k$-dimensional random vectors
distributed as $\mathrm{N}_k(\mathbf{0},\bmD)$ and $\|\bmD\|_{(2,2)}
\leq K$. Then for the sample variance $\bmS=n^{-1}\sumin\bZ_i\bZ_i^T$,
we have for $r \in\{2,\infty\}$
%
\begin{equation}
\P\left[\left\|\bmS- \bmD\right\|_{(r,r)} \geq t \right] \leq M k^2
\exp(-mnk^{-2}t^2), \quad|t| \leq m',
\end{equation}
where $M,m,m'>0$ depend on $K$ only.

In particular, if $k^2(\log k)/n\to0$, then $\|\bmS\|_{(\infty
,\infty
)}=O_P(1)$.
\end{lemma}

\begin{proof}
By the estimate of the large deviation probability given in Lemma~A.3
of \cite{bickel2008regularized}, it is immediate that $\P[\|\bmS-
\bmD
\|_{\iy} \geq t] \leq M k^2 \exp(-mnt^2)$.
Now the assertion follows by noting from Lemma \ref{lemmaA1} that $\|
\bmS- \bmD\|_{(r,r)} \leq k\|\bmS- \bmD\|_{\infty}$.
\end{proof}

\begin{lemma}
\label{lemmaA4}
Let $\bZ_i$, $i=1,\ldots,n$, be i.i.d. $k$-dimensional random vectors
distributed as $\mathrm{N}_k(\mathbf{0},\bmD)$ and $\max\{\|\bmD
^{-1}\|_{(r,r)}, \|\bmD\|_{(2,2)} \}\leq K$ for $r \in\{
2,\infty\}$. Then for the sample variance $\bmS=n^{-1}\sumin\bZ_i\bZ_i^T$,
we have
%
\begin{equation}
\P\left[\|\bmS^{-1}\|_{(r,r)} \geq M\right] \leq M'k^2 \exp(-mnk^{-2}C'^2),
\end{equation}
where $M>K$ and $M',m>0$ depend on $M$ and $K$ only.
\end{lemma}

\begin{proof}
Note that,
%
\begin{eqnarray}
\|\bmS^{-1}\|_{(r,r)} &\leq& \|\bmD^{-1}\|_{(r,r)} + \|\bmS^{-1} -
\bmD
^{-1}\|_{(r,r)} \nonumber\\
&=& \|\bmD^{-1}\|_{(r,r)} + \|\bmD^{-1}\|_{(r,r)}\|\bmS-\bmD\|_{(r,r)}
\|\bmS^{-1}\|_{(r,r)} \nonumber\\
&\le& K(1+\|\bmS-\bmD\|_{(r,r)} \|\bmS^{-1}\|_{(r,r)}).
\end{eqnarray}
This implies that
\begin{equation*}
\|\bmS^{-1}\|_{(r,r)} \leq\frac{K}{1-\|\bmS-\bmD\|_{(r,r)}K}.
\end{equation*}
Thus, using Lemma \ref{lemmaA3}, we obtain
%
\begin{equation}
\begin{split}
\P\left[\|\bmS^{-1}\|_{(r,r)} \geq M\right] &\leq\P\left[\frac
{K}{1-\|
\bmS-\bmD\|_{(r,r)}K} \geq M\right] \\
&\leq\P\left[\left\|\bmS- \bmD\right\|_{(r,r)} \geq K^{-1}-M^{-1}
\right] \\
&\leq M' k^2\exp(-m nk^{-2}).
\end{split}
\end{equation}

\end{proof}

\bibliographystyle{imsart-nameyear}

\end{document}